\documentclass[11pt]{amsart}
\usepackage{amsmath,amstext,amsbsy, amssymb, amscd, amsxtra,amsthm, color, amsfonts}
\usepackage{graphicx, longtable}
\usepackage[enableskew,vcentermath]{youngtab}
\usepackage{caption, multirow}

\newtheorem{thm}{Theorem}[section]
\theoremstyle{plain}

\newtheorem{lem}[thm]{Lemma}
\newtheorem{prop}[thm]{Proposition}
\newtheorem{cor}[thm]{Corollary}

\newtheorem{conj}[thm]{Conjecture}

\theoremstyle{definition}
\newtheorem{defn}[thm]{Definition}
\newtheorem{example}[thm]{Example}

\theoremstyle{remark}
\newtheorem{remark}[thm]{Remark}

\numberwithin{equation}{section}

\newcommand{\ch}{\text{ch}}
\newcommand{\C}{\mathbb C}
\newcommand{\Cl}{{\mathcal Cl}}
\newcommand{\la}{\lambda}
\newcommand{\La}{\Lambda}
\newcommand{\sgn}{\text{sgn}}
\newcommand{\tr}{\text{tr}}
\newcommand{\B}{\mathcal B} 
\newcommand{\OP}{\mathcal{OP}}
\newcommand{\EP}{\mathcal{EP}}
\newcommand{\calP}{\mathcal{P}}

\newcommand{\SOP}{\mathcal{SOP}}
\newcommand{\typeM}{\texttt M}
\newcommand{\typeQ}{\texttt Q}
\newcommand{\al}{\alpha}
\newcommand{\be}{\beta}
\newcommand{\Z}{\mathbb Z}
\newcommand{\h}{V}

\newcommand{\mf}{\mathfrak}

\newcommand{\mc}{\mathcal}
\newcommand{\Q}{\mathbb Q}
\newcommand{\Hom}{\text{Hom} }
\newcommand{\supp}{\text{supp} }

\newcommand{\D}[1][\la]{D^{\{#1,#1'\}}}
\newcommand{\ev}[1]{{#1}_{\bar{0}}}
\newcommand{\od}[1]{{#1}_{\bar{1}}}

\newcommand{\wtd}{\widetilde}
\newcommand{\td}{\tilde}

\newcommand{\negone}{\mbox{\tiny$-1$}}
\newcommand{\negtwo}{\mbox{\tiny$-2$}}
\newcommand{\zero}{\mbox{\tiny$0$}}
\newcommand{\one}{\mbox{\tiny$1$}}
\newcommand{\two}{\mbox{\tiny$2$}}
\newcommand{\three}{\mbox{\tiny$3$}}
\newcommand{\four}{\mbox{\tiny$4$}}
\newcommand{\six}{\mbox{\tiny$6$}}

\newcommand{\aHC}{\mathfrak{H}^{\mathfrak c}}

\title[Spin fake degrees for classical Weyl groups]
{Coinvariant algebras and fake degrees for spin Weyl groups of classical type}

\author[Baltera and Wang]{Constance Baltera and Weiqiang Wang}

\address{Department of Mathematics, University of Virginia, Charlottesville, VA 22904}
\email{cgb2k@virginia.edu (Baltera), \quad ww9c@virginia.edu (Wang)}

\begin{document}

\maketitle

\begin{abstract}
The coinvariant algebra of a Weyl group plays a fundamental role in
several areas of mathematics. The fake degrees are the graded
multiplicities of the irreducible modules of a Weyl group in its
coinvariant algebra, and they were computed by Steinberg, Lusztig
and Beynon-Lusztig. In this paper we formulate a notion of spin
coinvariant algebra for every Weyl group.  Then we compute all the spin
fake degrees for each classical Weyl group, which are by definition the graded multiplicities of
the simple modules of a spin 
 Weyl group in the spin coinvariant
algebra. The spin fake degrees for the exceptional Weyl groups are given in a sequel.
\end{abstract}

\let\thefootnote\relax

\setcounter{tocdepth}{1}
\tableofcontents

\section{Introduction}

\subsection{Background}

Let $V$ be the irreducible reflection representation of a Weyl group
$W$. The invariant algebra  $(S^*V)^W$ and the coinvariant algebra
$(S^*V)_W$ of $W$ are fundamental objects which have connections and
applications in many areas of geometry and representation theory.
According to Chevalley, the coinvariant algebra $(S^*V)_W$ is a
graded regular representation of $W$ (see \cite{Hu, Lu2}). Following
Lusztig, the graded multiplicity of a simple $W$-character $\rho$ in
$(S^*V)_W$ is called the fake degree of $\rho$,  and it is a
polynomial in a variable $t$ which specializes at $t=1$ to the
degree of $\rho$. The fake degrees were computed by Steinberg
\cite{Stn} in type $A_{n}$ (where $W$ is the symmetric group
$S_{n+1}$), by Lusztig \cite{Lu1} for type $B_n$ and $D_n$, and by
Beynon-Lusztig \cite{BL} using computer calculations for the
exceptional types. The formulation and computation of fake degrees
have significant applications to finite groups of Lie type, which
were systematically developed by Lusztig \cite{Lu1, Lu2}.

We start with a distinguished double cover $\wtd{W}$ for any Weyl
group $W$:
\begin{eqnarray}  \label{ext}
1 \longrightarrow \{1, z\} \longrightarrow \wtd{W} \longrightarrow W
\longrightarrow 1.
\end{eqnarray}
Schur \cite{Sch} in 1911 computed the Schur multiplier for the
symmetric groups $S_n$ and initiated the spin representation theory
of $S_n$; see \cite{Joz2} for a clear new exposition based on a
systematic use of superalgebras.
The Schur multiplier of $W$ has been computed by Ihara and Yokonuma
\cite{IY} (cf. \cite{Kar}) to be $\Z_2$ or a product of multiple
copies of $\Z_2$. The double cover $\wtd W$ in \eqref{ext} appeared in Morris \cite{Mo2} and it
corresponds to the choice of $2$-cocycle with all elements
nontrivial in every copy of $\Z_2$. Another description of $\wtd W$
is as follows. Assume that $W$ is generated by $s_1,\ldots, s_n$
subject to the relations
$(s_{i}s_{j})^{m_{ij}}=1$ for all $i,j$. The double cover $\wtd W$
is chosen so that the {\em spin Weyl group algebra} $\C W^- := \C \wtd{W}/
\langle z+1 \rangle$ is generated by $t_1,\ldots, t_n$ subject
to the relations
$(t_{i}t_{j})^{m_{ij}}=(-1)^{1+m_{ij}}$. One notable feature of $\C
W^-$ is that it is naturally a superalgebra with each $t_i$ being odd.

\subsection{Goal}

The goal of this paper and its sequel \cite{BW} 
is to formulate and compute the {\em spin fake degrees} of all simple characters of $\C W^-$ 
(which are the graded multiplicities
in so-called spin coinvariant algebras), for every Weyl group
$W$; except type $A$ which was done and expressed in terms of a
shifted $q$-hook formula in \cite{WW1, WW2}. The
computation is carried out case-by-case. Neat closed
$q$-hook formulas of spin fake degrees are obtained for $W$ of
{\em classical} type in this paper; the spin fake degrees for
the {\em exceptional} types are tabulated in the sequel \cite{BW}.

\subsection{Formulation}

The first problem which we encounter is that no natural candidate for a graded regular 
representation of $\C W^-$ is immediately available. We get around the difficulty as follows.

The reflection representation $V$ of $W$ is naturally endowed with a
$W$-invariant bilinear form $(\cdot,\cdot)$. The Clifford
(super)algebra $\Cl_V$ associated to $(V, (\cdot,\cdot))$ is acted
upon by $W$ as automorphisms, and the semi-direct product $\Cl_\h
\rtimes W$ is naturally a superalgebra. Khongsap and the second
author \cite[Theorem 2.4]{KW} have established an isomorphism of
superalgebras:
\begin{equation} \label{eq:KW=}
\Phi: \Cl_\h \rtimes W \stackrel{\simeq}{\longrightarrow} \Cl_\h
\otimes \C W^-.
\end{equation}
In the case when $W$ is a symmetric group, this was established by
Sergeev \cite{Se} and Yamaguchi \cite{Ya}.  For 
non-crystallographic reflection groups we cannot make sense of $\Cl_V$ 
and \eqref{eq:KW=}, so we do not consider these
groups here or in  \cite{BW}.
Modules of a superalgebra $A$ are assumed to have a
$\Z_2$-graded structure compatible with the action of $A$ unless
specified otherwise. We shall denote by $|A|$ the
underlying algebra of $A$.

The Clifford algebra $\Cl_\h$ is a simple superalgebra, and hence the
isomorphism \eqref{eq:KW=} induces a Morita super-equivalence
between the superalgebras $\aHC_W:=\Cl_\h \rtimes W$ and $\C W^-$
(see Proposition~\ref{functors}), and the study of the representation
theory of $\C W^-$ is essentially equivalent to the counterpart for
$\aHC_W$. The  tensor superalgebra
\begin{equation} \label{eq:spincoinv}
\Cl_\h \otimes (S^*V)_W
\end{equation}
is naturally a graded regular
representation of $\aHC_W$, and hence will be called the spin
coinvariant algebra. (This goes back to Wan and the second author
\cite{WW1} for $W=S_n$.) The graded multiplicity of a simple
$\aHC_W$-character $\chi$ in $\Cl_\h \otimes (S^*V)_W$ will be
called the spin fake degree of $\chi$ and denoted by $P(\chi,t)$.

Under the Morita super-equivalence induced by $\Phi$ in \eqref{eq:KW=}, the
simple $\aHC_W$-module $\Cl_\h$ is shown to correspond to the basic
spin $\C W^-$-module $\B_W$
(see Theorem~\ref{th:basicMQ}), and $\Cl_\h \otimes (S^*V)_W$ corresponds to 
$\B_W \otimes (S^*V)_W$. Here the basic spin module $\B_W$ is the
pullback of the simple $\Cl_\h$-module via a homomorphism $\C W^-$
to $\Cl_\h$ \cite{Mo2}, and the construction goes back to Schur for
$W=S_n$ (cf. \cite{Joz2}).  The graded multiplicity of a simple $\C
W^-$-character $\chi^-$ in $\B_W \otimes (S^*V)_W$ is called the
spin fake degree of $\chi^-$ and denoted by $P^-(\chi^-,t)$. It is
shown that $P(\chi,t)$ and $P^-(\chi^-,t)$ essentially coincide, up
to a possible factor of $2$ which is determined by Proposition~\ref{prop:equiv}, when
$\chi$ corresponds to $\chi^-$ under the super-equivalence.

\subsection{Main results}
To simplify notation, we will denote by $X_n$ the Weyl group of type
$X_n$ and the associated spin group algebra by $\C X_n^-$ (except
that in type $A$ we write $\C S_n^-$); for example $\C B_n^-$
denotes the spin Weyl group algebra of type $B_n$.
For $W$ of type $B_n$ or $D_n$, the split classes of $W$ (with
respect to $\wtd W$) were classified and the simple ungraded $|\C
W^-|$-modules were all constructed by Read \cite{Re2}. 

By the foundational work in the module theory of superalgebras
developed by J\'ozefiak \cite{Joz1} (also cf. \cite[Chapter~3]{CW}),
the numbers of even and odd split conjugacy classes determine the
numbers of simple $\C W^-$-modules of type $\typeM$ and type
$\typeQ$. So, we need to determine which split conjugacy classes
given by Read are even or odd. Fortunately, the parity of
a split class can be determined easily by the parity of the number of
generators in a representative of the given split class.

In this paper we classify the simple $\C W^-$-modules, not just the
ungraded simple ones. This turns out to be a subtle problem which
requires a combination of ideas and approaches case-by-case. To that
end, we establish some structure theorems for the superalgebras $\C
W^-$ in type $B_n$ and type $D_n$. More precisely, we establish the
superalgebra isomorphisms (see Theorems~\ref{CBn-iso} and \ref{Dnodd iso}):
\begin{align}  \label{eq:2isomBD}
\C B_n^-  \stackrel{\cong}{\longrightarrow} \Cl_n \otimes
\C S_n \quad (\forall n),
\qquad
\C D_n^-  \stackrel{\cong}{\longrightarrow} \Cl_{n}^0 \otimes \C
S_n \quad (n \text{ odd}),
\end{align}
where we denote
$\Cl_{n}^0$ the even subalgebra of $\Cl_n$ (we also formulate a conjecture on $\mathbb CD_n^-$ 
for $n$ even). The first isomorphism in
\eqref{eq:2isomBD} is obtained by reinterpreting \cite[Theorem~
1]{KW3} for $S_n$, where the role of $\C B_n^-$ was not suspected.
The construction and classification of simple $\C W^-$-modules
immediately follow from such an isomorphism. For $D_n$
with $n$ even, we find a simple argument to upgrade Read's results
\cite{Re2}.

We in addition calculate the characters of all simple $\C
B_n^-$-modules, and establish a characteristic map similar to the
one by Frobenius which relates symmetric group representations to the
ring of symmetric functions. We also provide a similar construction
and classification for the superalgebra $\aHC_{B_n}$. This allows us
to compute a simple precise formula for the spin fake degrees in
type $B_n$ in terms of a specialization of the super Schur functions
and also in terms of the hook lengths and contents of a Young diagram;
see Theorem~\ref{sfd:Bn} and Theorem~\ref{sfd:HCBn}.

Note that $\C D_n^-$ can be regarded naturally as a subalgebra of
$\C B_n^-$; see \cite[4.1]{KW3}.  We determine in a precise way how
each simple $\C B_n^-$-module decomposes upon restriction into the
simple $\C D_n^-$-modules, depending on whether $n$ is odd or even.
With this available, the spin fake degrees of type $D_n$ can then be
derived from those of type $B_n$; see Theorems~\ref{PDn-} and
\ref{PDn-:even}.

All the spin fake degrees for all Weyl groups are shown to be
palindromic, although the proof for the exceptional groups is deferred to the sequel. 
A similar palindromicity was observed for the usual
fake degrees by Beynon-Lusztig \cite{BL}.

\subsection{Connections}

The formulation of spin coinvariant algebras \eqref{eq:spincoinv}
associated to the spin Weyl groups $\wtd W$ has its origin in
\cite{KW} (and \cite{WW1}), where the main goal was to develop spin Hecke algebras
\cite{KW, KW2, KW3}. The spin (affine nil) Hecke algebras have
recently played a basic role in categorification of quantum
supergroups.
The same double covers $\wtd W$ also featured naturally in the
recent work of  Barbasch, Ciubotaru, and Trapa
\cite{BCT} in connection with Springer correspondence and
affine Hecke algebras. It would be very interesting to understand
why exactly the same spin Weyl groups appear in such diverse settings
and to develop any possible connections.

In Lusztig's work \cite{Lu2}, the fake degrees were related to the
generic degrees arising from Hecke algebras. In type $A$, the
generic degrees coincide with the fake degrees \cite{Stn}. Recently,
the spin generic degrees were formulated and computed in terms of
(quantum) spin Hecke algebras of type $A$ \cite{WW3}, and they were
shown to coincide with the spin fake degrees of $S_n$ (computed in
\cite{WW1}). The quantum spin Hecke algebras beyond type $A$ have
yet to be formulated.

Brou\'e-Malle-Michel 
and others (see \cite{BMM} and references therein) have attempted to generalize to the setting of
complex reflection groups various connections among Weyl groups,
Hecke algebras and finite groups of Lie type. Our work can be
formally regarded as a step toward generalization in the direction of spin Weyl
groups.

\subsection{Organization}

The paper is organized as follows.

The  preliminary Section~\ref{sec:finite}  reviews the double covers
$\wtd W$ of Weyl groups and some basics on the module theory of
superalgebras.

In Section~\ref{sec:sfd}, we formulate the spin coinvariant
algebras, and define the spin fake degrees for the superalgebras $\C
W^-$ and $\aHC_W$. We formulate Morita super-equivalence of
superalgebras, and show that the basic spin $\C W^-$-module
corresponds to the $\aHC_W$-module $\Cl_V$ via a Morita
super-equivalence (see Theorem~\ref{th:basicMQ}). The results 
of Section~\ref{sec:sfd} are valid
for both classical and exceptional Weyl groups.  

In Section~\ref{sec:Bn}, the first isomorphism for $\C B_n^-$ in
\eqref{eq:2isomBD} is established. We construct and classify the
simple $\C B_n^-$-modules, compute their characters, and establish a
characteristic map. Then we reduce the computation of the spin fake
degrees for simple $\C B_n^-$-modules to their counterparts for
simple $\aHC_W$-modules, which is carried out in
Section~\ref{sec:HCBn}.

In Section~\ref{sec:Dn:odd} where $n$ is set to be odd, the second
isomorphism for $\C D_n^-$ in \eqref{eq:2isomBD} is established.
The simple $\C D_n^-$-modules are constructed and classified. The
relation between simple modules of $\C B_n^-$ and $\C D_n^-$ is
worked out precisely. This allows us to reduce the computation of
spin fake degrees for $D_n$ to the counterparts for $B_n$.

Section~\ref{sec:Dn:even} on spin fake degrees of $D_n$ for $n$ even
is the counterpart of Section~\ref{sec:Dn:odd} (which was for $n$
odd), though the detail depends much on the parity of $n$.

\section{The preliminaries}
\label{sec:finite}

In this preliminary section, we review various facts on spin Weyl groups and 
semisimple superalgebras for later use.

\subsection{Notation} 

We let $\calP$ denote the set of all partitions, $\OP$ the set of all odd partitions, $\EP$ 
the set of all even partitions, and $\SOP$ the set of all strict odd partitions, i.e., 
those odd partitions containing no repeated parts.  When we wish to consider only 
partitions of a given $n$, we add a subscript; thus $\mc P_n$ is the set of partitions of $n$.  
Additionally, we let $P_n^{\rm{od}}$ be the set of partitions of $n$ of odd length, $\mc
P_n^{\rm{ev}}$ the set of partitions of $n$ of even length, and $\mc
P_n^{\rm{sym}}$ the set of (conjugate-)symmetric partitions of $n$.
Given partitions $\alpha,\beta$, we denote by $\alpha\cup\beta$ the
partition obtained by collecting and rearranging the parts from
$\alpha$ and $\beta.$

For a partition $\la=(\la_1, \la_2, \ldots)$ of $n$, we write $|\la| = n$ and denote $\la \vdash n$.  Additionally, 
\begin{eqnarray}   \label{n lambda}
n(\lambda)=\sum_{i\geq 1}(i-1)\lambda_i.
\end{eqnarray}
We  denote by  $h_\square$ the {\em hook length}  associated to a
cell $\square$ in the Young diagram of $\la$; the {\em content}
associated to a cell $\square$ is defined to be the difference
between the column number and the row number of $\square$.
See the following example.
\begin{example} 
Let 
$\la =(4,3,1)$. Then, $n(\la) =5$, the hook lengths are listed in the corresponding
cells of the left-hand diagram, and the contents in the
corresponding cells of the right-hand diagram as follows:
$$
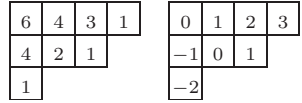
\young(\six\four\three\one,\four\two\one,\one) \quad
\young(\zero\one\two\three,\negone\zero\one,\negtwo)
$$
\end{example}

\vspace{.3cm}
 
A {\em module} over a superalgebra $ A=  A_{\bar 0} \oplus
 A_{\bar 1}$
is always understood in this paper as a $\Z_2$-graded $
A$-module $M = M_{\bar 0} \oplus M_{\bar 1}$ whose grading is
compatible with the action of $ A$, i.e.~$ A_iM_j
\subseteq M_{i+j}$. We shall denote by $| A|$ the underlying
algebra of $A$ with $\Z_2$-grading forgotten, and by $|M|$ the
$| A|$-module with  $\Z_2$-grading of $M$ forgotten.

\subsection{Weyl groups} \label{weyl gps}Let $W$ be an (irreducible)
finite Weyl group with the following presentation:
\begin{eqnarray} \label{eq:weyl}
\langle s_1,\ldots,s_n | (s_is_j)^{m_{ij}} = 1,\ m_{i i} = 1,
 \ m_{i j} = m_{j i} \in \Z_{\geq 2}, \text{for }  i
 \neq j \rangle.
\end{eqnarray}
In the case of type $A_{n}$, $W$ is the symmetric group $S_{n+1}$.
For a Weyl group $W$, the integers $m_{i j}$ take values in
$\{1,2,3,4,6\}$, and they are specified by the following
Coxeter-Dynkin diagrams whose vertices correspond to the generators
of $W$. By convention, we only mark the edge connecting $i,j$ with
$m_{ij} \ge 4$. We have $m_{ij}=3$ for $i \neq j$ connected by an
unmarked edge, and $m_{ij}=2$ if $i,j$ are not connected by an edge.

 
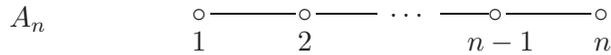
\begin{equation*}
 \begin{picture}(150,45) 
 \put(-99,18){$A_{n}$}
 \put(-30,20){$\circ$}
 \put(-23,23){\line(1,0){32}}
 \put(10,20){$\circ$}
 \put(17,23){\line(1,0){23}}
 \put(41,22){ \dots }
 \put(64,23){\line(1,0){18}}
 \put(82,20){$\circ$}
 \put(89,23){\line(1,0){32}}
 \put(122,20){$\circ$}
 \put(-30,9){$1$}
 \put(10,9){$2$}
 \put(74,9){${n-1}$}
 \put(122,9){${n}$}
 \end{picture}
 \end{equation*}
 %
 \begin{equation*}
 \begin{picture}(150,55) 
 \put(-99,18){$B_{n}(n\ge 2)$}
 \put(-30,20){$\circ$}
 \put(-23,23){\line(1,0){32}}
 \put(10,20){$\circ$}
 \put(17,23){\line(1,0){23}}
 \put(41,22){ \dots }
 \put(64,23){\line(1,0){18}}
 \put(82,20){$\circ$}
 \put(89,23){\line(1,0){32}}
 \put(122,20){$\circ$}
 \put(-30,10){$1$}
 \put(10,10){$2$}
 \put(74,10){${n-1}$}
 \put(122,10){${n}$}
 %
 \put(102,24){$4$}
 \end{picture}
 \end{equation*}
%
%
 \begin{equation*}
 \begin{picture}(150,75) 
 \put(-99,28){$D_{n} (n \ge 4)$}
 \put(-30,30){$\circ$}
 \put(-23,33){\line(1,0){32}}
 \put(10,30){$\circ$}
 \put(17,33){\line(1,0){15}}
 \put(35,30){$\cdots$}
 \put(52,33){\line(1,0){15}}
 \put(68,30){$\circ$ }
 \put(75,33){\line(1,0){32}}
 \put(108,30){$\circ$}
 \put(113,36){\line(1,1){25}}
 \put(138,61){$\circ$}
 \put(113,29){\line(1,-1){25}}
 \put(138,-1){$\circ$}
 \put(-29,20){$1$}
 \put(10,20){$2$}
 \put(60,20){$n-3$}
 \put(117,30){$n-2$}
 \put(145,0){$n-1$}
 \put(145,60){$n$}
 %
 \end{picture}
 \end{equation*}
%
%
 \begin{equation*}
 \begin{picture}(150,75) 
 \put(-99,28){$E_{n=6,7,8}$}
 \put(-30,30){$\circ$}
 \put(-23,33){\line(1,0){32}}
 \put(10,30){$\circ$}
 \put(17,33){\line(1,0){32}}
 \put(50,30){$\circ$}
 \put(57,33){\line(1,0){18}}

 \put(81,32){ \dots }
 \put(104,33){\line(1,0){18}}
 \put(122,30){$\circ$}
 \put(129,33){\line(1,0){32}}
 \put(162,30){$\circ$}
 \put(50,-8){$\circ$}
 \put(53,-1){\line(0,1){32}}
 \put(-30,39){$1$}
 \put(10,39){$3$}
 \put(50,39){$4$}
 \put(110,39){$n-1$}
 \put(162,39){$n$}
 \put(50,-17){$2$}
 \end{picture}
 \end{equation*}
%
%
 \begin{equation*}
 \begin{picture}(150,75) 
 \put(-99,28){$F_4$}
 \put(-30,30){$\circ$}
 \put(-23,33){\line(1,0){32}}
 \put(10,30){$\circ$}
 \put(17,33){\line(1,0){32}}
 \put(50,30){$\circ$}
 \put(57,33){\line(1,0){32}}
 \put(90,30){$\circ$}
 \put(-30,20){$1$}
 \put(10,20){$2$}
 \put(50,20){$3$}
 \put(90,20){$4$}
 \put(30,35){$4$}
 \end{picture}
 \end{equation*}
%
%
 \begin{equation*}
 \begin{picture}(150,55) 
 \put(-99,28){$G_2$}
 \put(-30,30){$\circ$}
 \put(-23,33){\line(1,0){32}}
 \put(10,30){$\circ$}
 \put(-30,20){$1$}
 \put(10,20){$2$}
 \put(-10,35){$6$}
 \end{picture}
 \end{equation*}
\subsection{Spin Weyl groups}
\label{subsec:spinWeyl}

The Schur multipliers for finite Weyl groups $W$ have been computed
by Ihara and Yokonuma \cite{IY}. The explicit generators and
relations for the corresponding covering groups of $W$ can be found
in Karpilovsky \cite[Table 7.1]{Kar}.

In this paper (as in \cite{KW}), we shall be concerned exclusively
with a distinguished double covering $\wtd{W}$ of $W$:
\begin{equation}
\label{eq:ses}
 1 \longrightarrow \Z_2 \longrightarrow \wtd{W}
\stackrel{\theta}{\longrightarrow} W \longrightarrow 1.
\end{equation}
We denote by $\Z_2 =\{1,z\},$ and by $\td{t}_i$ a fixed preimage
of the generators $s_i$ of $W$ for each $i$. The group $\wtd{W}$
is generated by $z, \td{t}_1,\ldots, \td{t}_n$ with relations

%
\begin{equation}  \label{eq:wtdW:rel}
z^2 =1, \qquad
 (\td{t}_{i}\td{t}_{j})^{m_{ij}} =
 \left\{
\begin{array}{rl}
1, & \text{if } m_{ij}=1,3  \\
z, & \text{if }  m_{ij}=2,4,6.
\end{array}
\right.
\end{equation}

The quotient algebra of $\C \wtd{W}$ by the ideal generated by $z+1$
is denoted by $\C W^-$ and called the {\em spin Weyl group algebra}
associated to $W$. Denote by $t_i \in \C W^-$ the image of
$\td{t}_i$. The spin Weyl group algebra $\C W^-$ has the following
uniform presentation: $\C W^-$ is the algebra generated by $t_i,
1\le i\le n$, subject to the relations
\begin{equation}
(t_{i}t_{j})^{m_{ij}} = (-1)^{m_{ij}+1}.
\end{equation}
The algebra $\C W^-$ has a natural
superalgebra structure by letting  each  $t_i$ be odd.

\begin{example}  \label{present}
Let $W$ be the Weyl group of type $A_{n},B_{n},$ or $D_{n}$. Then the spin Weyl group
algebra $\C W^-$ is generated by $t_1,\ldots, t_n$ with the
labeling as in the Coxeter-Dynkin diagrams and the explicit
relations summarized in the following Table~A.
%
 \begin{center}
{Table A: Relations for classical spin Weyl group algebras}
  \vspace{.1cm}
\begin{tabular}
[t]{|l|l|}\hline Type of $W$ & Defining Relations for $\C W^-$\\
\hline $A_{n}$  & $t_{i}^{2}=1$, \quad
$t_{i}t_{i+1}t_{i}=t_{i+1}t_{i}t_{i+1}$ if $1\le i\le n$,\\ &
$(t_{i}t_{j})^{2}=-1\text{ if } |i-j|\,>1$\\
\hline   & $t_{1},\ldots,t_{n-1}$ satisfy the relations for $\C W^-_{A_{n-1}}$, \\
$B_{n}$ &  $t_{n}^{2}=1,\quad (t_{i}t_{n})^{2}=-1$ if $i\neq n-1,n$, \\
&
$(t_{n-1}t_{n})^{4}=-1$\\
\hline    & $t_{1},\ldots,t_{n-1}$ satisfy the relations for
$\C W^-_{A_{n-1}}$,\\
$D_{n}$&  $t_{n}^{2}=1,\quad (t_{i}t_{n})^{2}=-1$ if $i\neq n-2, n$,
\\& $t_{n-2}t_{n}t_{n-2}=t_{n}t_{n-2}t_{n}$\\\hline
\end{tabular}
\end{center}
\end{example}

\bigskip%

\subsection{Clifford algebra}
\label{subsec:cliff}

Denote by $\h$ the irreducible reflection representation
of dimension $n$ of the Weyl group $W$
(which is the Cartan subalgebra of the corresponding simple Lie algebra).

Note that $\h$ carries a $W$-invariant nondegenerate bilinear form
$(\cdot,\cdot)$, and let
$\Cl_\h$ be the Clifford algebra associated to $(\h, (\cdot,\cdot))$.
Denote by $\be_i$ the generator of $\Cl_\h$ corresponding to the
simple root $\al_i$ normalized with $\be_i^2=1$. Note that $\Cl_\h$
is naturally a superalgebra with each $\be_i$ being odd.
We identify $\h$
with a suitable subspace of $\C^m$ (for values of $m$ see Table~B below),
and then describe the simple roots $\{\alpha_i\}$ for $\mathfrak g$
using a standard orthonormal basis $\{e_i\}$ of $\C^m$. It follows
that $(\alpha_i, \alpha_j) =-2\cos (\pi /m_{ij})$.
Let $\Cl_m$ denote the Clifford algebra of $\C^m$ which is generated by
$c_{1},\ldots,c_m$ subject to the relations
\begin{align}  \label{clifford}
c_{i}^{2} =1,\qquad c_{i}c_{j} =-c_{j}c_{i}\text{ if } i\neq j.
\end{align}
(Here $c_i$ corresponds to the basis element $e_i$.)
It is convenient to identify $\Cl_V$
as a subalgebra of $\Cl_m$ (see Table~B); we may also identify $\Cl_V$ with $\Cl_n$ 
and shall do so whenever convenient.


 \begin{center}
  \vspace{.2cm}
 {Table B: Generators for Clifford algebra $\Cl_\h$}
  \vspace{.2cm}
\begin{tabular}
[t]{|l|l|p{3.5in}|}\hline Type of $W$&$m$ & Generators for
$\Cl_\h$\\
\hline $A_{n}$ & $n+1$&
$\be_{i}=\frac{1}{\sqrt{2}}(c_{i}-c_{i+1}),1\leq i\leq n$\\
\hline $B_{n}$ &$n$&
$\be_{i}=\frac{1}{\sqrt{2}}(c_{i}-c_{i+1}),1\leq
i\leq n-1$, $\be_{n}=c_{n}$\\
\hline $D_{n}$ &$n$ &
$\be_{i}=\frac{1}{\sqrt{2}}(c_{i}-c_{i+1}),1\leq
i\leq n-1$, $\be_{n}=\frac{1}{\sqrt{2}}(c_{n-1}+c_{n})$\\
\hline
$E_{8}$ &8& $\be_{1}=\frac{1}{2\sqrt{2}}(c_{1}+c_{8}-c_{2}-c_{3}%
-c_{4}-c_{5}-c_{6}-c_{7})$\\ &&
$\be_{2}=\frac{1}{\sqrt{2}}(c_{1}+c_{2}),\be_{i}=\frac
{1}{\sqrt{2}}(c_{i-1}+c_{i-2}),3\leq i\leq8$\\
\hline $E_{7}$ &8& the subset of $\be_{i}$ in $E_{8}$, $1\le i\le
7$
\\
\hline $E_{6}$ &8& the subset of $\be_{i}$ in $E_{8}$, $1\le i\le 6$
\\
\hline $F_{4}$ &4&
$\be_{1}=\frac{1}{\sqrt{2}}(c_{1}-c_{2}),\be_{2}=\frac{1}{\sqrt{2}}(c_{2}-c_{3})$
\\
&& $\be_{3}=c_{3},\be_{4}=\frac{1}{2}(c_{4}-c_{1}-c_{2}%
-c_{3})$\\
\hline $G_{2}$ &3& $\be_{1}
=\frac{1}{\sqrt{2}}(c_{1}-c_{2}),\be_{2}=\frac{1}{\sqrt{6}}(-2c_{1}+c_{2}+c_{3})$
\\
\hline
\end{tabular}
\end{center}
\bigskip%

The action of $W$ on $\h$ preserves the bilinear form
$(\cdot,\cdot)$ and thus $W$ acts as automorphisms of the algebra
$\Cl_\h$. This gives rise to a semi-direct product
$$
\aHC_W :=\Cl_\h \rtimes  W,
$$
which is called the {\em Hecke-Clifford algebra} for $W$. The
algebra $\aHC_W$ naturally inherits the superalgebra
structure by letting elements in $W$ be even and each $\be_i$ be
odd.

\subsection{Simple modules of superalgebras}

In this subsection, we shall recall some standard facts
about semisimple superalgebras from \cite{Joz1}
(cf.  \cite{Kle} or \cite{CW}).

The space of all $(r+s) \times (r+s)$ matrices, denoted by $M(r|s)$, is a
superalgebra with the following grading, with the matrices expressed
in $(r, s)$-block form:
$$
M(r|s)_{\bar 0} = \left\{ \left(\begin{array}{c|c}A & 0 \\\hline 0 &
D\end{array}\right)\right\}, \quad M(r|s)_{\bar 1} = \left\{
\left(\begin{array}{c|c}0 & B \\\hline C &
0\end{array}\right)\right\}.
$$
The set of $2n \times 2n$ matrices $Q(n) =
\left\{\left(\begin{array}{c|c}A & B \\\hline B &
A\end{array}\right)\right\},$ with $A, B$ arbitrary $n\times n$
matrices, is a subalgebra of the superalgebra $M(n|n)$.

Both $M(r|s)$ and $Q(n)$ are  simple superalgebras.
A classical theorem due to Wall states that
all finite-dimensional simple associative superalgebras over $\C$
are isomorphic to either $M(r|s)$ or $Q(n)$, for suitable $r, s$ or $n$.

\begin{thm}[Super Wedderburn's Theorem]\label{wedderburn}
Let $ A$ be a finite-dimensional semisimple associative
superalgebra.  Then $$
A \cong \bigoplus_{i=1}^m M(r_i|s_i)
\oplus \bigoplus_{j=1}^q Q(n_j).
$$
\end{thm}

As in the ungraded case, each simple $ A$-module will be
annihilated by all but one of the simple summands in
Theorem~\ref{wedderburn}.  Since there are two types of simple
superalgebras, there will also be two types of simple $A$-modules.
We say that a simple $A$-module is {\em of type
$\typeM$} if the summand which does not annihilate it is of the form
$M(r_i|s_i)$, and {\em of type $\typeQ$} if this summand is of the
form $Q(n_j)$.  The following generalization of Schur's lemma
distinguishes between them.

\begin{lem}[Super Schur's Lemma]\label{super schur}
Let $M$ and $L$ be simple $ A$-modules.  Then
$$
\dim \Hom_A (M, L) =
\left\{\begin{array}{lll}
1 & \mbox{
if }  & M \cong L  \mbox{ of type } \typeM,
\\
2 & \mbox{ if } & M
\cong L  \mbox{ of type } \typeQ,
\\
0  & \mbox{ if } & M \ncong L.
\end{array}\right.
$$
\end{lem}

\begin{remark}  \label{rem:ungraded}
Let $A$ be a finite-dimensional $\C$-superalgebra.
A type $\typeM$ simple $A$-module
remains simple as an $|A|$-module, while a type $\typeQ$ simple
$A$-module becomes a sum of two non-isomorphic simple $|A|$-modules;
see \cite{Joz1}.
\end{remark}

\subsection{Split conjugacy classes}

We now consider the conjugacy classes of $W$ and $\wtd W$.  All the
elements of a given conjugacy class have the same parity, so we can
describe each conjugacy class in $W$  as either even or odd.

Let $K$ be a conjugacy class of $W$.  Then $\theta^{-1} (K)$ is
either a single conjugacy class of $\wtd W$, or splits into two as
$\theta^{-1}(K) = \wtd K \sqcup z\wtd K$; in the latter case, we say
that $K$, $\wtd K$, and $z\wtd K$ are {\em split} classes.  We say
$x \in W$ is {\em split} (which actually depends on
$\wtd W$) if it belongs to a split conjugacy class.
If we denote $\theta^{-1}(z) = \{\tilde x, z\tilde x\}$, $x$ is
split if and only if $\tilde x$ and $z\tilde x$ are not conjugate in
$\wtd W$.

\begin{prop}\cite[Proposition 4.14]{Joz1} \label{prop:splitcc}
The number of even split conjugacy classes of $W$ is equal to the
total number of simple $\C W^-$-modules. The number of odd split
conjugacy classes is equal to the number of simple $\C W^-$-modules
of type $\typeQ$.
\end{prop}

\section{Spin coinvariant algebras and spin fake degrees}
\label{sec:sfd}

In this section we formulate the notion of spin coinvariant algebras
and then the spin fake degrees. The Morita
super-equivalence between the spin Weyl group algebras and the
Hecke-Clifford algebras plays an essential role. Throughout this section the formulations
and results are valid for arbitrary Weyl groups. 

\subsection{Spin coinvariant algebras}

Let $A$ be a superalgebra. We shall denote by $A\text{-}\mf{mod}$
the category of (finite-dimensional) modules of the superalgebra $A$
(with morphisms of degree one allowed). There is a parity reversing
functor
$$
\Pi: A\text{-}\mf{mod} \longrightarrow A\text{-}\mf{mod},
$$
which sends $M=\ev M +\od M$ to $\Pi M$ with $\ev{(\Pi M)}=\od M$
and $\od{(\Pi M)}=\ev M$. The underlying even subcategory
$A\text{-}\mf{mod}_{\bar 0}$, which consists of the same objects as $A\text{-}\mf{mod}$ but
only even morphisms, is an abelian category. We define
the Grothendieck group $R(A)$ of the category $A\text{-}\mf{mod}$ to
be the $\Z$-module generated by all objects in $A\text{-}\mf{mod}$
subject to the following two relations: (i)
$[\Pi M]=[M]$, (ii) $[M]=[L] +[N]$, for all $L,M,N$ in $A\text{-}\mf{mod}$ satisfying a
short exact sequence $0\rightarrow L\rightarrow M\rightarrow
N\rightarrow 0$ with {\em even} morphisms.

Given two superalgebras $A$ and $B$, we view the
tensor product of superalgebras $A$ $\otimes$
$B$ as a superalgebra with multiplication defined by
\begin{equation}  \label{eq:tensorsuper}
(a\otimes b)(a^{\prime}\otimes b^{\prime})
=(-1)^{|b||a^{\prime}|}(aa^{\prime }\otimes bb^{\prime})\text{ \ \ \
\ \ \ \ }(a,a^{\prime}\in{A},\text{
}b,b^{\prime}\in{B})
\end{equation}
where $|b|$ denotes the $\Z_2$-degree of $b$, etc. Also, we shall
use short-hand notation $ab$ for $(a\otimes b) \in A$ $\otimes$ $B$,
$a = a\otimes1$, and $b=1\otimes b$.
We recall the following isomorphism, which extends an earlier result
of Sergeev \cite{Se} and Yamaguchi \cite{Ya} in type $A$ (cf. e.g. \cite{WW2}).

\begin{prop} \label{isofinite} \cite[Theorem 2.4]{KW}
We have an isomorphism of superalgebras:
$$
\Phi: \Cl_\h \rtimes  W
\stackrel{\simeq}{\longrightarrow} \Cl_\h \otimes \C W^-,
$$
which extends the identity map on $\Cl_\h$ and sends $s_i \mapsto
-\sqrt{-1} \be_i t_i, \forall i.$ The inverse map $\Psi$ is the
extension of the identity map on $\Cl_\h$ which sends $ t_i \mapsto
\sqrt{-1} \be_i s_i, \forall i.$
\end{prop}
As before, we shall refer to $\aHC_W = \Cl_\h \rtimes W$ as
the {\em Hecke-Clifford (super)algebra}.
As in Section \ref{subsec:cliff}, the Weyl group $W$ acts on $V$
as its reflection representation, and then on the symmetric
algebra $S^*V$. The
coinvariant algebra is $(S^*V)_W =S^*V/\langle(S^*V)^W_+\rangle$ by definition, where
$\langle(S^*V)^W_+\rangle$ denotes the ideal generated by the
homogeneous $W$-invariants of positive degrees. A classical theorem
of Chevalley states that $(S^*V)_W =\bigoplus_k (S^k V)_W$ is a graded regular
representation of $W$ (cf. \cite{Hu}).

\begin{defn}  \label{def:coinvalg}
The {\em spin coinvariant algebra} for $W$ is defined to be
$\Cl_V\otimes (S^*V)_W$.
\end{defn}

Note that
$$
\Cl_V\otimes (S^*V)_W=\bigoplus_k \Cl_V \otimes
(S^k V)_W
$$
is a
graded regular representation of the Hecke-Clifford superalgebra
$\aHC_W$, where $\Cl_V$ acts by left multiplication on the first
tensor factor and $W$ acts diagonally, and this justifies the
terminology in Definition~\ref{def:coinvalg}.
More generally, given a $W$-module $M$, $\Cl_V\otimes
M$ is naturally an $\aHC_W$-module.

\subsection{Morita super-equivalence}

Recall that  $\Cl_n$ is a simple superalgebra, with a unique (up to
isomorphism) irreducible module $U$.   The module $U$ is of type
$\texttt M$ for $n$ even and of type $\texttt Q$ for $n$ odd. We
have $\dim U=2^{k}$ for $n=2k$ or $n=2k-1$.

Assume that a superalgebra
isomorphism $\Cl_n\otimes A \cong B$ exists for
two superalgebras $A$ and $B$.  Then the two exact functors
\begin{eqnarray}  \label{eq:functorFG}
\begin{split}
\mathfrak{F} \stackrel{\text{def}}{=}U\otimes -: & A\text{-}\mf{mod}
\longrightarrow B\text{-}\mf{mod},
 \\
\mathfrak{G}  \stackrel{\text{def}}{=} {\rm Hom}_{\Cl_n}(U,-): &
B\text{-}\mf{mod} \longrightarrow A\text{-}\mf{mod},
\end{split}
\end{eqnarray}
define a Morita super-equivalence in the following sense
(cf. \cite[Proposition~13.2.2]{Kle}).

\begin{prop}\label{functors}
Assume that two superalgebras $A, B$ satisfy a superalgebra
isomorphism $\Cl_n\otimes A \cong B$. Let $\mf F, \mf G$ be defined
as in \eqref{eq:functorFG}.

\begin{enumerate}
\item Suppose that $n$ is even. Then the two functors
$\mathfrak{F}$ and $\mathfrak{G}$ are equivalences of categories
such that $
\mathfrak{F}\circ\mathfrak{G}\cong{\rm id},\;
\mathfrak{G}\circ\mathfrak{F}\cong{\rm id}.
$

\item Suppose that $n$ is odd. Then
$\mathfrak{F}\circ\mathfrak{G}\cong{\rm id}\oplus\Pi,\;
\mathfrak{G}\circ\mathfrak{F}\cong{\rm id}\oplus\Pi. $
Moreover, $\mathfrak{F}$ induces a bijection between the
isoclasses of irreducible $A$-modules of type \texttt{M} and the
isoclasses of irreducible $B$-modules of type \texttt{Q}.
Also $\mathfrak{G}$ induces a bijection between the
isoclasses of irreducible $B$-modules of type \texttt{M} and
the isoclasses of irreducible $A$-modules of type \texttt{Q}.
\end{enumerate}
\end{prop}
In particular, the Hecke-Clifford algebra $\aHC_W =
\Cl_\h \rtimes W$ and the spin Weyl group algebra $\C W^-$ are Morita super-equivalent
by Proposition~\ref{isofinite}.

\subsection{The basic spin module}

In \cite{Mo2}, Morris studied the same double cover $\wtd{W}$ as in \eqref{eq:ses},
and showed that there exists a surjective superalgebra
homomorphism
\begin{equation} \label{eq:morris}
\Omega: \C W^-  \longrightarrow \Cl_\h, \qquad t_i \mapsto \be_i
\;\forall i.
\end{equation}
By pulling back the unique simple module $U$ of the Clifford superalgebra
$\Cl_\h$ via
the homomorphism $\Omega$, we obtain a distinguished $\C
W^-$-module, called the {\em basic spin
module}, which we shall denote by $\B_W$.
This is a natural generalization of the classical construction for
$\C S_n^-$ due to Schur \cite{Sch} (see \cite{Joz3}).

The character of the basic spin module $U$ of
$\Cl_n$ (with standard generators $c_1,\ldots, c_n$)
will be useful in later computations; we recall
it here.  Let $c_I = c_{i_1}\dots c_{i_p}$ be associated with an (ordered)
subset $I = \{i_1,\dots,i_p\}$ of $\{1,\ldots,n\}$, and $c_\emptyset =1$.
\begin{prop}\cite[Section 3C]{Joz2} \label{charU}
The character value of $U$ at $c_I$ is equal to
$$
\left\{
\begin{array}{lll}
0 & \mbox{ if } & I \ne \emptyset,
\\
2^{k} & \mbox{ if } & I =
\emptyset, \text{ and } n=2k \text{ or } 2k+1.
\end{array}
\right.
$$
\end{prop}

The following property of basic spin modules plays
a fundamental role in the formulation of the notion of
spin fake degrees. Though we only need the case of classical Weyl groups in this paper,
we have included the exceptional type so that we do not need to repeat much of the setup
in the sequel \cite{BW}. 

\begin{thm}  \label{th:basicMQ}
Let $W$ be an arbitrary (classical or exceptional) Weyl group, with $V$ its irreducible
reflection representation. Then
\begin{enumerate}
\item
The basic spin $\C W^-$-module $\mc B_W$ is simple, of type $\typeM$
if $\ \dim V$ is even and of type $\typeQ$ if $\ \dim V$ is odd.

\item  $\mf G (\Cl_\h) \cong \B_W$ as $\C W^-$-modules.

\item
$\Cl_\h$ is a simple $\aHC_W$-module always of type $\typeM$.
\end{enumerate}
\end{thm}

\begin{proof}
Since $\Omega: \C W^-\rightarrow \Cl_V$ in \eqref{eq:morris} is surjective, the
$\C W^-$-module $\B_W$,  as
the pullback of the simple $\Cl_V$-module $U$ via $\Omega$,
must be simple and its type comes from the type of the simple
$\Cl_V$-module $U$, whence (1).

Part~(3) follows immediately by (2) and
Proposition~\ref{functors}.

So it remains to prove (2). The proof is case-by-case, and
there are 2 main approaches: one via character computation and
the other by dimension counting.

The first approach is to verify by a character computation that as
$\aHC_W$-modules,
\begin{align}  \label{eq:FBw}
\mf F(\B_W) \cong
 \begin{cases}
  \Cl_V,  & \text{ if  $\dim V$ is even},
   \\
  \Cl_V^{\oplus 2},  & \text{ if  $\dim V$ is odd}.
\end{cases}
\end{align}
Indeed for type $B_n$, this isomorphism is a special case of Lemma~\ref{lem:K=B}
below (where $\la$ is a one-row partition $(n)$).
The verification for types $A$ and $D$ can also be read off
from the proof of Lemma~\ref{lem:K=B}, since $\C S_n^-$ and $\C D_n^-$
are naturally subalgebras of $\C B_n^-$.

Since $\Cl_V$ is a simple superalgebra with simple module $U$,
$\mf G(\Cl_V) ={\rm Hom}_{\Cl_V} (U, \Cl_V)$ has dimension equal to $\dim U$
(which is the same as $\dim \B_W$).
Then Part~(2) for a given Weyl group $W$ is valid if the following holds for $W$:
\begin{align}  \label{eq:dimcount}
& \B_W  \text{ is the unique simple } \C W^-\text{-module of minimal
 dimension}, \\
& \text{ and the minimal dimension is equal to }\dim U. \nonumber
\end{align}

It turns out \eqref{eq:dimcount} holds for exceptional Weyl groups
$E_6, E_7,$ and $E_8$, according to the degrees of all spin simple
characters computed by Morris \cite{Mo1}; alternatively, it can be
read off from Tables for $E_6$, $E_7$, and $E_8$ in
\cite{BW}. Actually it can also be easily observed
that \eqref{eq:dimcount} holds for $B_n$ and $D_n$ from the
construction and classification of the simple $\C W^-$-modules in
later sections (see Propositions~\ref{spinBnirreps},
\ref{spinDnoddirreps}, \ref{spinDnevenirreps}). This gives a second
proof in types $B_n$ and $D_n$. But we do not know how to check
\eqref{eq:dimcount} directly for type $A_n$, though the degrees of
simple characters are well known since Schur (cf. \cite{Joz2}).

However, \eqref{eq:dimcount} is not true for $G_2$ and $F_4$. In
these two cases, we verify \eqref{eq:FBw} by a direct character
computation as follows. We shall freely use \cite{Mo1}  (see also \cite{BW}).

The three simple spin characters of $\C G_2^-$ are all of degree 2
and type $\typeM$, and they have different values on the conjugacy
class with admissible diagram $G_2$ (for which we can choose
$t_1t_2$ as a representative;  cf. \cite[Section 3, ex.~(ii)]{Ca}).
Hence, to show \eqref{eq:FBw} (with $\dim V=2$) it suffices to check
that both sides of \eqref{eq:FBw} take the same (nonzero) character
value at $s_1s_2$. We refer to Table~B for formulas of $\beta_1$ and
$\beta_2$. The action of $s_1s_2$ on $\mf F (\B_W) =U\otimes \B_W$
is given by $\Phi (s_1s_2) =\beta_1\beta_2\cdot t_1t_2$. The trace
of $\beta_1\beta_2
=-\frac{\sqrt3}2-\frac1{\sqrt12}(c_1c_2-c_1c_3+c_2c_3)$
on $U$ is $-\sqrt 3$.
Since $\Omega(t_1t_2) =\beta_1\beta_2$, we see that
the trace of $t_1t_2$ on $\B_W$ is also $-\sqrt 3$.   (Note this is the opposite of
the value given in \cite[Table VI]{Mo1}, as we have made a different
choice of $\C \wtd G_2$ conjugacy class in the preimage of the $\C
G_2^-$ conjugacy class in question.) Hence
the character value of $s_1s_2$ on $\mf F(\B_W)$ is $(-\sqrt 3)^2=3$.

On the other hand, the character of $s_1s_2$ on $\Cl_V$ is also $3$
by the following computation:
 \allowdisplaybreaks{
\begin{align*}
s_1s_2.1& = 1,
 \\
s_1s_2.\beta_1 &
 = s_1(\beta_1 +\sqrt 3 \beta_2)
 =2\beta_1 +\sqrt{3} \beta_2,
  \\
s_1s_2.\beta_2 & = - s_1\beta_2   = -\sqrt 3 \beta_1 -\beta_2,
   \\
s_1s_2.\beta_1\beta_2 &  = \beta_1\beta_2.
\end{align*}}
Hence \eqref{eq:FBw} holds for $G_2$.
%

Now consider the case of $F_4$. Following Morris \cite{Mo1} and Read
\cite{Re1}, there are  two simple spin characters of $\C F_4^-$ of
minimal degree 4 (both of type $\typeM$); they have opposite
character values on the conjugacy class with admissible diagram
$B_2$. Read \cite[Table~ 1]{Re1} provides the representative element
$t_2t_3$ for this conjugacy class, and we will use this element to
compute character values.

We identify $\Cl_V$ with $\Cl_4$, and refer to Table~B for formulas
of $\beta_i$.  Since $\Omega(t_2t_3)
=\beta_2\beta_3$, we see that the trace of $t_2t_3$ on $\B_{F_4}$ is
equal to the trace of $\beta_2\beta_3=\frac1{\sqrt2}c_2c_3 -
\frac 1{\sqrt2}$ on $U$, which is $-2\sqrt 2$. (Note this is the opposite of the value given in
\cite[Table~ VII]{Mo1}, as we have made a different choice of the
split conjugacy class in question.  It is, however, the same as the
value given in \cite[Table~1]{Re1}.) Thus the character of $\frak
F(\B_{F_4}) =U\otimes \B_{F_4}$ on $t_2t_3$ is $(-2\sqrt{2}) \dim U
=-8\sqrt 2$.

On the other hand, we consider
$$
\Psi(t_2t_3) = -\beta_2s_2\beta_3s_3 =
-(\beta_2\beta_3 + \sqrt2)s_2s_3 \in \aHC_{F_4}.
$$
A direct yet lengthy computation shows that the matrix of the operator
$\Psi(t_2t_3)$ acting on $\Cl_V$ with respect to the ordered basis
\begin{align*}
& \{1,\beta_1, \beta_2, \beta_3, \beta_4,
\beta_1\beta_2,\beta_1\beta_3,\beta_1\beta_4, \beta_2\beta_3,
\beta_2\beta_4,
 \\
& \qquad\qquad \beta_3\beta_4, \beta_1\beta_2\beta_3,
\beta_1\beta_2\beta_4, \beta_1\beta_3\beta_4, \beta_2\beta_3\beta_4,
\beta_1\beta_2\beta_3\beta_4\}
\end{align*}
has its diagonal given by
$$
\text{diag}\ (-\sqrt2, -\sqrt2, -\sqrt2, 0,-\sqrt2, -\sqrt2, 0,-\sqrt2, 0,
-\sqrt2, 0, 0, -\sqrt2, 0,0,0).
$$
Thus the character of $\Cl_V$ on
$\Psi(t_2t_3)$ is $-8\sqrt2$, agreeing with $\frak F(\B_{F_4})$.
Hence \eqref{eq:FBw} holds for $F_4$.

The proof of (2) and hence of the proposition is now completed.
\end{proof}

\begin{remark}
Let $W=S_n$. If we choose to work with the reflection representation
$\C^n$ which is not irreducible as in \cite{Se, Ya, KW}, $\Cl_n$
is a simple $(\Cl_n\rtimes S_n)$-module, now {\em of type $\typeQ$}.
Theorem~\ref{th:basicMQ} in such a setting was stated without proof
in \cite{WW1}.
\end{remark}

\subsection{A multiplicity identity}

Let $M$ be a $W$-module, $E$ a $\C W^-$-module, and $F$ an
$\aHC_W$-module.  Then the tensor product $E\otimes M$ is a $\C
W^-$-module under the action
\begin{equation}\label{CW- tensor}
t_i (u \otimes x) = (t_iu) \otimes (s_ix) \qquad 1\le i \le n, u \in
E, x\in M.
\end{equation}
Additionally, the tensor product $F\otimes M$ is an $\aHC_W$-module
via
\begin{equation}\label{aHC_W tensor}
\beta_i(u\otimes x) = (\beta_i u) \otimes x, \quad s_i(u\otimes x) = (s_iu)
\otimes (s_i x)
\end{equation}
for $1\le i\le n, u\in F, x\in M$.

The following tensor identity is a straightforward generalization of
\cite[Lemma~ 3.1]{WW1}, and it can be proved in the same way.

\begin{lem}  \label{lem:tensor}
Let $M$ be a $W$-module.  Then there is a $\C W^-$-module
isomorphism $$\mathfrak G(\Cl_V) \otimes M \cong \mathfrak G (\Cl_V
\otimes M);$$ that is, $\Hom_{\Cl_V}(U, \Cl_V) \otimes M \cong
\Hom_{\Cl_V}(U, \Cl_V\otimes M).$
\end{lem}
%
%
%

Using the Morita super-equivalence of Proposition~\ref{functors} in the
context of Proposition~\ref{isofinite}, we can relate the
multiplicity problem for a $\C W^-$-module and that for a
$\aHC_W$-module as follows.
We will abuse notation to sometimes use module and
character names interchangeably, and denote a module and its
associated character by the same notation.

\begin{prop} \label{prop:equiv}
Suppose $M$ is a $W$-module.  Let $\chi$ be a simple
$\aHC_W$-character, and $\chi^-$ the corresponding simple
$\C W^-$-character under the Morita super-equivalence.
Let $m_\chi = \dim \Hom_{\aHC_W}(\chi, \Cl_V\otimes M)$
and $m_\chi^- = \dim \Hom_{\C W^-}(\chi^-, \B_W\otimes M)$.  Then
$$
m_\chi^- =
\left\{
    \begin{array}{lcl}
    m_\chi &  & \mbox{ if $n$ is even},
     \\
    2 m_\chi &  & \mbox{ if $n$ is odd and $\chi$ is of type $\typeM$},
     \\
    m_\chi &  & \mbox{ if $n$ is odd and $\chi$ is of type $\typeQ$}.
   \end{array}
\right.
$$
\end{prop}

\begin{proof}
By definition and Lemma \ref{super schur}, we have
\begin{align} 
\label{eq:mult}
\Cl_\h \otimes M &= \bigoplus_{\chi \mbox{ type } \typeM} m_\chi \chi
\oplus \bigoplus_{\chi\mbox{ type } \typeQ} \frac 12 m_\chi \chi,
  \\
\label{eq:mult-}
\B_W \otimes M &= \bigoplus_{\chi^- \mbox{ type } \typeM} m_\chi^-
\chi^- \oplus \bigoplus_{\chi^-\mbox{ type } \typeQ} \frac 12
m_\chi^- \chi^-.
\end{align}
Hence, by Proposition~\ref{functors}, Theorem~\ref{th:basicMQ},
Lemma~\ref{lem:tensor} and \eqref{eq:mult},
\begin{align*}
\B_W \otimes M &= \mathfrak G(\Cl_\h)\otimes M
 \\
& \cong \mathfrak G(\Cl_\h \otimes M)
 \\
& \cong \bigoplus_{\chi \mbox{ type }
\typeM} m_\chi \mathfrak G(\chi) \oplus \bigoplus_{\chi\mbox{ type }
\typeQ} \frac 12 m_\chi \mathfrak G(\chi)
 \\
& \cong \left\{
 \begin{array}{cc}
\bigoplus_{\chi \mbox{ type } \typeM} m_\chi \chi^- \oplus
\bigoplus_{\chi\mbox{ type } \typeQ} \frac 12 m_\chi \chi^- & \mbox{
if $n$ is even} \\ \\
\bigoplus_{\chi \mbox{ type } \typeM} m_\chi \chi^-
\oplus \bigoplus_{\chi\mbox{ type } \typeQ} m_\chi \chi^- & \mbox{
if $n$ is odd.}
\end{array}\right.
\end{align*}
Comparing this with \eqref{eq:mult-} gives the result desired.
\end{proof}

\subsection{Spin fake degrees}

Let $\chi$ be a simple character of $\aHC_W$, and
let $\chi^-$ be a simple character of $\C W^-$
corresponding to $\chi$ under the Morita super-equivalence
as in Proposition~\ref{functors}. 
Let $t$ be an indeterminate; then we define
\begin{align}  \label{eq:sfdegree}
\begin{split}
P_W(\chi, t)  &= \displaystyle \sum_k \dim \Hom_{\aHC_W} (\chi,
\Cl_\h \otimes (S^kV)_W) t^k,
 \\
P_W^-(\chi^-, t) &= \displaystyle \sum_k \dim \Hom_{\C W^-} (\chi^-,
\B_W \otimes (S^kV)_W) t^k;
\end{split}
\end{align}
\begin{align}  \label{eq:Hmult}
\begin{split}
H_W(\chi, t) &= \displaystyle \sum_k \dim \Hom_{\aHC_W} (\chi,
\Cl_\h \otimes S^kV) t^k,
  \\
H_W^-(\chi^-, t) &= \displaystyle \sum_k \dim \Hom_{\C W^-} (\chi^-,
\B_W \otimes S^kV) t^k.
\end{split}
\end{align}
We can reformulate the above definitions in terms of the bilinear form $(\cdot,\cdot)$ and the formal sum
$$
S_tV := \sum_{j\ge 0} (S^jV)t^j.
$$
For example, $H_W^-(\chi^-, t) = \dim \Hom_{\C W^-} (\chi^-, \B_W\otimes S_tV)$.
We will refer to $H_W(\chi, t)$ informally as the graded
multiplicity of $\chi$ in the $\aHC_W$-module $\Cl_\h \otimes S^*V$,
and refer to $H_W^-(\chi^-, t)$ as the graded multiplicity of
$\chi^-$ in the $\C W^-$-module $\B_W \otimes S^*V$.

\begin{defn}
$P_W(\chi, t)$ is called the {\em spin fake
degree} of the simple $\aHC_W$-character  $\chi$, and
$P_W^-(\chi^-, t)$ is called the {\em spin
fake degree} of the simple $\C W^-$-character $\chi^-$.
\end{defn}

\begin{remark}The fake degrees of a Weyl group $W$ are the
graded multiplicities of simple $W$-modules in the coinvariant
algebra of $W$ (cf. Lusztig \cite{Lu2}).

The spin coinvariant algebra and spin fake degrees for the symmetric
group $S_n$ were first formulated and computed by Wan and the second
author \cite{WW1}, and the terminology of spin fake degrees first
appeared in \cite{WW2}.
\end{remark}

Let $d_1, \dots, d_n$ be the degrees of the Weyl group $W$ (cf. \cite{Hu, Lu2});
we recall their values in the classical cases in Table~C.
\begin{center}
\vspace{.2cm}
{Table C: Degrees of classical Weyl groups $W$}

\vspace{.2cm}
\begin{tabular}{|l|c|c|c|}
\hline
Type & $A_{n}$ & $B_n$ & $D_n$ \\
\hline
Degrees & $2,3,\dots, n+1$ & $2,4,\dots, 2n$ & $2,4,\dots, 2n-2, n$\\
\hline
\end{tabular}
\end{center}
\bigskip
Recall (cf. \cite{Hu})  the algebra of $W$-invariants in $S^*V$ is a polynomial algebra
whose Hilbert polynomial is given by
\begin{equation}   \label{eq:Hinv}
H \big((S^*V)^W, t \big) =\frac1{\prod_{i=1}^n
 (1-t^{d_i})}.
\end{equation}

\begin{lem}  \label{lem:P=W}
Let $\chi$ be a simple $\aHC_W$-character and $\chi^-$
be a simple $\C W^-$-character.
Then
$P_W(\chi, t) =H_W(\chi, t) \prod_{i=1}^n(1-t^{d_i})$, and
$P_W^-(\chi^-, t) =H_W^-(\chi^-, t) \prod_{i=1}^n(1-t^{d_i}).$
\end{lem}

\begin{proof}
Follows by definition (see \eqref{eq:sfdegree}, \eqref{eq:Hmult}
and \eqref{eq:Hinv})
and the classical theorem  of Chevalley (cf. \cite{Hu}) that
$S^*V \cong (S^*V)^W \otimes (S^*V)_W$ as graded $W$-modules.
\end{proof}

The main goal of this paper and its sequel is to compute the spin fake degrees
$P_W(\chi, t)$ and $P_W^-(\chi^-, t)$ for every Weyl group $W$.
Lemma~\ref{lem:P=W} allows us to do the computations for the series
$H_W(\chi, t)$ and $H_W^-(\chi^-, t)$ instead.
Proposition~\ref{prop:equiv} allows us to transfer back and forth
any computation between $H_W(\chi, t)$ and $H_W^-(\chi^-, t)$. The
computations of all these multiplicities, which are formulated for
simple (graded) $\C W^-$-modules, can be readily translated into the
multiplicities of simple (ungraded) $|\C W^-|$-modules (with some
possible factors of $2$ which can be determined case-by-case).

\subsection{Palindromicity}

Let us summarize a symmetry property shared by all spin fake degrees
for all Weyl groups below. We thank Ching Hung Lam for a very
helpful remark.
\begin{thm}
For any Weyl group $W$, the spin fake degrees for $\C W^-$ are
palindromic. More precisely,
 for every irreducible character $\chi^-$ of $\ \C W^-$, we have
$$
P_W^-( \chi^-, t) = t^N P_W^-(\chi^-, t^{-1}),
$$
where $N$ is the number of reflections in
the Weyl group $W$.
\end{thm}

The values of $N$ here for each Weyl group are recalled
in Table~D below.

\begin{center}
  \vspace{.2cm}

{Table D: Number $N$ of reflections in $W$}
  \vspace{.2cm}

\begin{tabular}{|c|cccccccc|}
\hline
Type & $A_{n}$ & $B_n$ & $D_n$ & $E_6$ & $E_7$ & $E_8$ & $F_4$ & $G_2$\\
\hline
$N$ & ${n(n+1)}/2$&$n^2$&$n(n-1)$&36&63&120&24&6\\
\hline
\end{tabular}
  \vspace{.2cm}
\end{center}

\begin{proof}
The proof requires a case-by-case check; see Propositions
\ref{Bpalindrome}, \ref{Dpalindromeodd}, and \ref{Dpalindromeeven}, for 
the classical types other than type $A$; they rely
on the explicit formulas of spin fake degrees in these cases, which
we will compute in the subsequent sections. For the exceptional types, 
the theorem follows from the computation in \cite{BW}. 

Now let $W=S_n$ in type $A_{n-1}$. The simple $\C S_n^-$-modules are
parameterized by strict partitions $\la$ of $n$, and the
corresponding spin fake degrees $P_{S_n}^-(\la, t)$ were computed in
\cite[Theorem~A]{WW1} (cf. \cite{WW2}). The formula for
$P_{S_n}^-(\la, t)$ is in terms of $n(\la)$ given in \eqref{n
lambda}, contents $c_\square$, and shifted hook lengths $h_\square^*$,
and we refer to {\em loc. cit.} for detailed definitions. It is
actually clear that $P_{S_n}^-(\la, t) = t^a P_{S_n}^-(\la, t^{-1})$
for some shift integer $a$, since one observes that $P_{S_n}^-(\la,
t)$ is a product of factors of the form $(1\pm t^*)^{\pm 1}$.  So it
remains to determine the shift number $a$, which is the sum of the
highest and the lowest powers of $t$ appearing in
$P_{S_n}^-(\la,t)$. Thus
\begin{align*}
a & = 2n(\la) + \frac{n^2+n-2}2 + \sum_{\square \in \la^*} (c_\square - h^*_\square)\\
& = 2n(\la) + \frac{n^2+n-2}2 -\big(2n(\la) + n-1\big) = \frac{n^2-n}2,
\end{align*}
which is  the number of reflections in $S_n$.
\end{proof}

\begin{remark}
A similar palindromicity property holds for the usual fake degrees,
see Beynon-Lusztig \cite[Proposition A]{BL}, which can be regarded
as a variant of Poincar\'e duality. However, the shift numbers for the
usual fake degrees depend on the irreducible characters (as well as
on the Weyl groups).
\end{remark}

\section{The spin fake degrees of type $B_n$}
\label{sec:Bn}

\subsection{Structure of the algebra $\C B_n^-$}

We shall simply write the Weyl group of type $B_n$ as $B_n$, its
double cover as $\wtd B_n$, and the spin Weyl group algebra as $\C
B_n^-$. Recall the generators $\beta_i$ for $\Cl_V$ from Table~B,
where $V$ is the reflection representation of a Weyl group $W$ (in this case $B_n$), 
and note that we can identify $\Cl_V$ and $\Cl_n$. The
following is a new formulation of Khongsap-Wang \cite[Theorem~
1]{KW3} in the case of $S_n$, which now describes the structure of
the superalgebra $\C B_n^-$.

\begin{thm} \label{CBn-iso}
There is an isomorphism of superalgebras
$$
\phi^B: \C B_n^- \stackrel{\cong}{\longrightarrow} \Cl_n \otimes
\C S_n,
$$
$$
t_i \mapsto  \left\{ \begin{array}{ll} \beta_is_i & \mbox{ if } i
\leq n-1,
 \\
 c_n & \mbox{ if } i = n.
\end{array}\right.
$$
The inverse map sends $s_i \mapsto \beta_it_i$ and $c_i \mapsto (-1)^{n-i}
t_it_{i+1}\cdots t_{n-1}t_nt_{n-1}\cdots t_{i+1}t_i$ for all
possible $i$. (Note each $s_i$ is even here.)
\end{thm}

\begin{proof}
Following \cite[Theorem 1]{KW3} we denote by $\Cl_n \rtimes_- \C
S_n^-$ the superalgebra generated by $\Cl_n$ and $\C S_n^-$ with the
additional relation that $t_i c_j = - c_j^{s_i}t_i$ for all $i, j$.
Our simple yet new observation here is that there is an isomorphism of superalgebras
\begin{equation}  \label{eq:identif}
\C B_n^-  \stackrel{\cong}{\longrightarrow}
   \Cl_n \rtimes_- \C S_n^-
\end{equation}
by sending $t_i \mapsto t_i$ for $i=1, \dots, n-1$, and $t_n \mapsto
c_n$. As the relations involving only $t_i$, $i=1, \dots, n-1$, are the same in both $\C B_n^-$ 
and $ \Cl_n \rtimes_- \C S_n^-$, we need only check that the relations involving $t_n$ are preserved.  

So we first confirm that $t_i c_n = - c_n t_i$ for $i \ne n-1, n$ in $\Cl_n \rtimes_-\C S_n^-$, 
which follows by the additional relation  $t_i c_j = - c_j^{s_i}t_i$.  Then we check that $t_{n-1}c_nt_{n-1}c_n =
-c_nt_{n-1}c_nt_{n-1}$ in $\Cl_n \rtimes_- \C S_n^-$. Indeed,
    \begin{align*}
    t_{n-1}c_nt_{n-1}c_n & = c_{n-1}t_{n-1}c_{n-1}t_{n-1}
     = -c_{n-1}c_n t_{n-1}t_{n-1} \\
    & = c_nc_{n-1}t_{n-1}t_{n-1}
     = -c_nt_{n-1}c_nt_{n-1}.
    \end{align*}
    
Note that this homomorphism is surjective.  Then injectivity follows by dimension counting.

On the other hand, we have an explicit isomorphism $\Cl_n \rtimes_-
\C S_n^- \cong \Cl_n \otimes \C S_n$, which extends the identity map
on $\Cl_n$ and sends $t_i \mapsto \beta_is_i$ for $i \leq n-1$, by
\cite[Theorem~ 1]{KW3} specialized for $W = S_n$. Now the
theorem follows from this isomorphism and the identification
\eqref{eq:identif}.
\end{proof}

\subsection{Split classes for $B_n$}


With the identification $\Z_2=\{+,-\}$,  an element $x$ in $B_n
=\Z^n_2\rtimes S_n$ is a product of positive and negative cycles of
various lengths. Collecting the lengths of positive (respectively, negative)
cycles together gives us a partition $\rho_+$ (respectively,
$\rho_-$), and we say $x$ is of type $(\rho_+,\rho_-)$. It is well
known \cite[I, Appendix~B]{Mac} that the conjugacy classes of the
group $B_n$ are parameterized by the types of total size $n$. For
example, the identity element of $B_n$ has type $((1^n), \emptyset)$.

\begin{lem} (cf. \cite{Re2})
 \label{lem:Bnsplitcc}
\begin{enumerate}
\item
The split conjugacy classes of $B_n$ are the classes of the
following types:
\begin{enumerate}
\item
$(\rho_+, \rho_-) \in (\OP, \EP)$;

\item
$(\rho_+, \rho_-) \in (\emptyset, \calP)$ \quad (only when $n$ is
odd).
\end{enumerate}
\item
The split classes of type $(\rho_+, \rho_-) \in (\OP, \EP)$ are even
while those of type $(\rho_+, \rho_-)\in (\emptyset, \calP)$ are
odd.
\end{enumerate}
\end{lem}

\begin{proof}
(1) is exactly \cite[Theorem~4.1]{Re2}, where Read uses the
terminology $\alpha$-regular to refer to split classes.

Since all generators $t_i$ are odd, the parity of an element
$t_{i_1}\ldots t_{i_k}$, and thus of its conjugacy class, is equal
to the parity of $k$. Now (2) follows by counting the number of
generators in a representative element of each conjugacy class as
given in \cite{Re2}.
\end{proof}

\subsection{Simple $\C B_n^-$-modules}

It follows from Proposition~\ref{prop:splitcc} and
Lemma~\ref{lem:Bnsplitcc} that all simple modules of $B_n$ for $n$
even (respectively, $n$ odd) are of type $\typeM$ (respectively,
type $\typeQ$). Denote by $S^\la$ the Specht module associated to a
partition $\la$. Recall the unique simple $\Cl_\h$-module $U$. The
pullback of the simple $(\Cl_\h \otimes \C S_n)$-module $U\otimes
S^\la$ via the isomorphism $\phi^B$ is a simple $\C
B_n^-$-module, which is denoted by $B^\la$.

\begin{prop}\label{spinBnirreps}
$\{B^\la  | \la \vdash n\}$ is a complete set of pairwise
inequivalent simple $\C B_n^-$-modules, all of type $\typeM$ when
$n$ is even, and all of type $\typeQ$ when $n$ is odd.
\end{prop}

\begin{proof}
This follows directly from the isomorphism in Theorem~\ref{CBn-iso};
note that the $\C S_n$ is purely even and that the $\C B_n^-$-module $U$
is of type $\typeM$ if and only if $n$ is even.
\end{proof}

\begin{remark}  \label{rem:ungrBn}
The {\em ungraded} simple modules for $\C B_n^-$ have been
classified and constructed in completely different approaches by
Read \cite[Theorem 5.1]{Re2} (also cf. Stembridge \cite[Theorem~
9.2]{St}). In light of Remark~\ref{rem:ungraded},
Proposition~\ref{spinBnirreps} allows us to recover easily Read's
classification of irreducible {\em ungraded} modules.
\end{remark}

Next we shall determine the character of $B^\la$.  We choose the
canonical positive cycle in $\C B_n^-$ which permutes $a$ through
$a+k$ to be $t_at_{a+1}\cdots t_{a+k-1}$, and the canonical negative
cycle in $\C B_n^-$ which permutes those same elements to be
$t_at_{a+1}\cdots t_{a+k-1}b_{a+k}$, where
\begin{equation}  \label{eq:bi}
b_n = t_n,\qquad b_i = t_i b_{i+1}t_i, \;\; \text{ for } 0\le i \le
n-1.
\end{equation}
(In particular, $b_i \mapsto (-1)^{n-i}c_i$ under the isomorphism in Theorem~
\ref{CBn-iso}.) The support of a (signed) permutation $\sigma \in B_n$ is
$\supp(\sigma) =\{i |1 \le i \le n, \sigma(i) \neq i\}$.  The representative element for the conjugacy class
of type $(\alpha, \beta)$ is the product of the corresponding
canonical cycles, chosen so that the supports of negative cycles are bigger than those of positive cycles. 
Let $\chi^\la_\mu$ be the character value of the
Specht module $S^\la$ evaluated on elements of $S_n$ of type $\mu$.
By Lemma~\ref{lem:Bnsplitcc}, the even split
conjugacy classes of $B_n$ are  parametrized by the types $(\alpha,
\beta) \in (\OP, \EP)$, which are in bijection with the partitions
of $n$ by taking $\alpha\cup \beta$; we will use such an
identification below whenever it is convenient.

\begin{prop}  \label{prop:charB}
The character value of the simple $\C B_n^-$-module $B^\la$ at an
even split element of type $(\alpha, \beta) \in (\OP, \EP)$ is
\begin{equation*}
\begin{cases}
 2^{\ell(\alpha\cup
\beta)/2}(-1)^{(n-\ell(\alpha))/2}\chi^\la_{\alpha\cup\beta} &
\mbox{ if $n$  is even,}
        \\
2^{(\ell(\alpha\cup\beta)+1)/2}
(-1)^{(n-\ell(\alpha))/2}\chi^\la_{\alpha\cup\beta} & \mbox{ if $n$
is odd.}
\end{cases}
\end{equation*}
\end{prop}

\begin{proof}
The proof is mainly based on Theorem~\ref{CBn-iso} and
Proposition~\ref{spinBnirreps}.

Let $x \in \C B_n^-$ be a representative element of type $(\alpha,
\beta) \in (\OP, \EP)$. When we compute the character value of
$B_\la$ at $x$, i.e., the character value of $U \otimes S^\la$ at
$\phi^B (x)\in \Cl_\h \otimes \C S_n$ (for the isomorphism
$\phi^B$ see Theorem~\ref{CBn-iso}), we may ignore all
nontrivial products of $c_i$ by Proposition~ \ref{charU}.

We write $x$ as a product of cycles.
For a positive $(k+1)$-cycle with $k$ even,
the image of the canonical cycle is
\begin{align*}
\phi^B (t_at_{a+1} \cdots t_{a+k-1})= & \beta_a\cdots
\beta_{a+k-1} s_a\cdots s_{a+k-1}
 \\
 = & 2^{-\frac k2} (c_ac_{a+1} - c_ac_{a+2} -1 + c_{a+1}c_{a+2})\\
 & \cdot (c_{a+2}c_{a+3}-c_{a+2}c_{a+4}-1+c_{a+3}c_{a+4}) \cdots
  \\
& \cdot (c_{a+k-2}c_{a+k-1}-c_{a+k-2}c_{a+k}-1+c_{a+k-1}c_{a+k})s_a \cdots s_{a+k-1}
 \\
= &2^{-\frac{k}2}(-1)^{\frac k2}s_a\cdots s_{a+k-1} + (\mbox{terms with $c_i$}).
\end{align*}

For a negative $(k+1)$-cycle with $k$ odd, we note that, since the positive cycles have 
supports in terms of smaller numbers than the negative cycles, we must have $n-a \equiv 1 \mod 2$.  So the image of the
canonical cycle is
\begin{align*}
\phi^B (t_at_{a+1}\cdots t_{a+k-1}b_{a+k})= & (-1)^{n-a-k}\beta_a\cdots
\beta_{a+k-1} c_{a+k}s_a\cdots s_{a+k-1}
\\
= &(-1)^{1-k} 2^{-\frac k2}(c_ac_{a+1} - c_ac_{a+2} -1 + c_{a+1}c_{a+2}) \cdots
 \\
&\cdot (c_{a+k-3}c_{a+k-2}-c_{a+k-3}c_{a+k-1}-1+c_{a+k-2}c_{a+k-1})
 \\
&\cdot (c_{a+k-1}c_{a+k}-1)s_a \cdots s_{a+k-1}
 \\
=&2^{-\frac
k2}(-1)^{\frac{k+1}2}s_a\cdots s_{a+k-1} + (\mbox{terms with
$c_i$}).
\end{align*}

Multiplying these together, we have
$$
\phi^B(x)
=2^{-\frac{n-\ell(\alpha\cup\beta)}{2}}(-1)^{\frac{n-\ell(\alpha)}2}
\sigma + (\mbox{terms with $c_i$}).
$$
Thus by Proposition \ref{charU}, the character value of $B^\la$ at
$x$, i.e., the character value of $U \otimes S^\la$ at
$\phi^B(x)$ is equal to
\begin{equation*}
\begin{cases}
2^{\frac
n2}2^{-\frac{n-\ell(\alpha\cup\beta)}{2}}(-1)^{\frac{n-\ell(\alpha)}2}
\chi^\la_{\alpha\cup\beta} & \mbox{ if $n$ is even,}
 \\
2^{\frac{n+1}2}2^{-\frac{n-\ell(\alpha\cup\beta)}{2}}(-1)^{\frac{n-\ell(\alpha)}2}
\chi^\la_{\alpha\cup\beta} & \mbox{ if $n$ is odd},
\end{cases}
\end{equation*}
which is the same as given in the proposition.
\end{proof}

\begin{remark}
Read \cite{Re2} chooses representative elements which differ from ours in the use 
of $(-1)^{n-i}b_i$ in place of $b_i$, but this difference in sign does not affect the 
computation of characters.  The character formula in Proposition~\ref{prop:charB} agrees with
that computed by Read \cite[Theorems~3.5, 5.1]{Re2}, and our
labeling of the simple (graded or ungraded) characters is consistent
with Read (cf. Remark~\ref{rem:ungrBn}). Stembridge \cite{St} used a
form of the basic spin module $\B_{B_n}$ resulting from $-\beta_i$
rather than $\beta_i$, and so his $\C B_n^-$-modules differ from
ours by a tensor with sgn.
\end{remark}

\subsection{The characteristic map for $\C B_n^-$}

Let $R(\C B_n^-)$ be the Grothendieck group of the category of $\C
B_n^-$-modules.  If we replace isomorphism classes of modules
by their characters, it becomes a free abelian group with a basis
made up of the irreducible characters. Now define
$$
R^- = \bigoplus_{n=0}^\infty R(\C B_n^-),
$$
when $R(\C B_0^-) = \mathbb Z$. Set $R^-_\Q :=\Q \otimes_\Z R^-$.

We shall define a ring structure on $R^-$ as follows. Let $\C
B_{m,n}^-$ be the subalgebra of $\C B^-_{m+n}$ generated by $\C
B^-_m \times \C B^-_n$. For a $\C B_m^-$-module $M$ and a $\C
B_n^-$-module $N$,  $M\otimes N$ is naturally a $\C
B^-_{m,n}$-module, and we define the product
$$
[M] \cdot [N] = [\C B^-_{m+n} \otimes_{\C B^-_{m,n}} (M \otimes N)],
$$
and then extend by $\Z$-bilinearity. It follows from the
properties of the induced characters that the multiplication on
$R^-$ is commutative and associative.

Given  $\C B^-_n$-modules $M,N$, we define a bilinear form on
$R^-$  by letting
\begin{equation}  \label{eq:bilformR-}
\langle M,N \rangle  =\dim \Hom_{\C B^-_n} (M,N).
\end{equation}

Denote by $\La$ the ring of symmetric functions in infinitely many
variables, which is the $\Z$-span of the monomial symmetric
functions $m_\la$ for $\la \in \mc P$, and let $\La_\Q =\Q\otimes_\Z \La$.
There is a standard bilinear
form $(\cdot, \cdot)$ on $\La$ such that the Schur functions $s_\la$
form an orthonormal basis for $\La$. The ring $\La_\Q$ admits several
other distinguished bases: the complete homogeneous symmetric functions
$\{h_\la\}$, the elementary symmetric functions $\{e_\la\}$, and the
power-sum symmetric functions $\{p_\la\}$. See \cite{Mac}.

Now define the {\em (spin) characteristic map} $\ch^- : R^-_\Q
\longrightarrow \La_\Q$ as the linear map
\begin{equation}   \label{eq:chmapB}
\ch^-(\phi) =\sum_{\la\vdash
n}z_\la^{-1}(-1)^{\frac{n-\ell(\alpha)}2}2^{-\frac{\ell(\la)}2}
\phi(\la) p_\la,
\end{equation}
where $\phi \in R(\C B_n^-)$, $\phi(\la)$ is the character value of
$\phi$ at an element of type $(\alpha,\beta)$ with $\alpha \cup
\beta = \la$ and  $\alpha \in \OP$ and $\beta \in \EP$, and $z_\la$
is the order of the centralizer in $S_n$ of an element of cycle type
$\la$.

Recall that, for $\mu \vdash n$, the Schur function $$s_\mu = \sum_{\la\vdash n} 
z_\la^{-1}\chi_\la^\mu p_\la,$$ where $\chi_\la^\mu$ is the character value of the 
Specht module $S^\mu$ on an element (of $S_n$) of cycle type $\la$.  

\begin{thm}
The characteristic map $\ch^-: R^-_\Q \rightarrow \La_\Q$ is an isometric
isomorphism of graded algebras, sending the character of $B^\la$ to
$s_\la$ when $|\la|$ is even and to $\sqrt2 s_\la$ when $|\la|$ is
odd.
\end{thm}

\begin{proof}
Recall that the characters of the irreducible modules $B^\la$ for $\la
\in \calP$, defined in Proposition~\ref{spinBnirreps}, form a basis
for $R^-$. This becomes an orthonormal basis if we divide the
characters of of type $\typeQ$ modules $B^\la$ by ${\sqrt 2}$,
thanks to the super version of Schur's Lemma, Lemma~
\ref{super schur}; this happens exactly when $n$ is odd by
Proposition~\ref{spinBnirreps}.

By plugging the character values of $B^\la$ computed in
Proposition~\ref{spinBnirreps} into \eqref{eq:chmapB},
the characteristic map $\ch^-$ sends the elements of this
orthonormal basis to the corresponding Schur functions, so it is an
isometry.

It remains to check that $\ch^-$ is an algebra homomorphism.  Let $\phi \in
R(\C B_m^-)$ and $\chi \in R(\C B_n^-)$, and consider the image of
their product under the characteristic map.  When splitting $\la =
(\alpha, \beta) \in(\OP,\EP)_{m+n}$ into partitions $\mu \vdash  m,
\nu \vdash n$, we will write $\alpha_\mu, \alpha_\nu$ for the
corresponding pieces of $\alpha$.
{\allowdisplaybreaks
\begin{align*}
\ch^-(\phi\chi)  = & \sum_{\lambda\vdash m+n}
z_\lambda^{-1}(-1)^{\frac{m+n-\ell(\alpha)}2}
2^{-\frac{\ell(\lambda)}2}(\phi\chi)(\la) p_\lambda
 \\
\stackrel{(\star)}{=} &\sum_{\lambda\vdash m+n}
z_\lambda^{-1}(-1)^{\frac{m+n-\ell(\alpha)}2}
2^{-\frac{\ell(\lambda)}2}\sum_{\substack{\mu \cup \nu = \lambda\\
\mu\vdash m, \nu\vdash n}} z_\lambda
z_\mu^{-1}z_\nu^{-1}\phi(\mu)\chi(\nu)p_\lambda
 \\
= &  \sum_{\lambda\vdash m+n}\sum_{\substack{\mu \cup \nu = \lambda\\
\mu\vdash m, \nu\vdash n}}
(-1)^{\frac{m-\ell(\alpha_\mu)}2}(-1)^{\frac{n-\ell(\alpha_\nu)}2}
2^{-\frac{\ell(\mu)}2} 2^{-\frac{\ell(\nu)}2}z_\mu^{-1}z_\nu^{-1}
\phi(\mu)\chi(\nu)p_\mu p_\nu
 \\
= &\sum_{\mu\vdash m}
z_\mu^{-1}(-1)^{\frac{m-\ell(\alpha_\mu)}2}2^{-\frac{\ell(\mu)}2}\phi(\mu)p_\mu
\sum_{\nu\vdash n} z_\nu^{-1}
(-1)^{\frac{n-\ell(\alpha_\nu)}2}2^{-\frac{\ell(\nu)}2}\chi(\nu)p_\nu
 \\
=  & \ch^-(\phi) \ch^- (\chi).
\end{align*}}
In the equation ($\star$) above, we have used a formula for the
character value $(\phi\chi)(\la)$, which can be established in a
completely analogous way as for Lemma~\ref{lem:phipsichar} below.
This proves the theorem.
\end{proof}

\subsection{Spin fake degrees of $B_n$}

%

Introduce a parity function
$$
p(n) =
 \begin{cases}
 0, & \text{ if } n \text{ is even}, \\
 1, & \text{ if } n \text{ is odd}.
 \end{cases}
$$

\begin{thm}  \label{thm:HBn-}
The graded multiplicity of $B^\la$ in the $\C B_n^-$-module $\B
\otimes S^*V$ is
$$
H^-_{B_n}(\la,t) = 2^{p(n)} t^{2n(\la)}\prod_{\square\in
\la}\frac{1+t^{2c_\square+1}}{1-t^{2h_\square}}.
$$
\end{thm}

\begin{proof}
We will compute the graded multiplicities for the simple
$\aHC_{B_n}$-modules $K^\la$ in $\Cl_n \otimes S^*V$ in Theorem~
\ref{thm: sfd HBn}.  The theorem follows from Lemma~\ref{lem:K=B},
Theorem~\ref{thm: sfd HBn} and Proposition~ \ref{prop:equiv}.
\end{proof}

The following is equivalent to Theorem~\ref{thm:HBn-} by
Lemma~\ref{lem:P=W} and use of the well-known fact that the degrees
of $B_n$ are $2,4, \ldots, 2n$.

\begin{thm}  \label{sfd:Bn}
The spin fake degree of $B^\la$, for $\la \vdash n$, is
$$
P_{B_n}^- (\la, t) = 2^{p(n)} t^{2n(\la)}\prod_{\square\in
\la}\frac{1+t^{2c_\square+1}}{1-t^{2h_\square}}(1-t^2)(1-t^4)\cdots(1-t^{2n}).
$$
\end{thm}

We have the following palindromicity
of the spin fake degrees for $B_n$.

\begin{prop} \label{Bpalindrome}
For $\la \vdash n$, we have
$
 P_{B_n}^- (\la, t) = t^{n^2} P^-_{B_n}(\la, t^{-1}).
 $
\end{prop}

\begin{proof}
Observe that $P_{B_n}^- (\la, t) = t^a P^-_{B_n}(\la, t^{-1})$ for
some integer $a$, since each of its factors is of the form $(1\pm
t^*)^{\pm 1}$.  It remains to determine the shift number $a$, which is the
sum of the highest power of $t$ appearing in $P_{B_n}^-(\la, t)$ with
nonzero coefficient, and the lowest.  Thus
\begin{align*}
a & = 2n(\la) + n(n+1) + \sum_{\square \in \la} (2c_\square + 1) -
2\sum_{\square \in \la} h_\square + 2 n(\la)
 \\
& = 4n(\la) + n^2 + 2n + 2\sum_{\square \in \la}(c_\square -
h_\square)
    \\
& = 4n(\la) + n^2 + 2n + 2(n(\la') - n(\la) -(n(\la) + n(\la') + n))
    = n^2.
\end{align*}
The proposition is proved.
\end{proof}

\section{The spin fake degrees of the Hecke-Clifford algebra $\aHC_{B_n}$}
\label{sec:HCBn}

In this section we will work with the Hecke-Clifford algebra in
order to complete the computation of spin fake degrees of type $B_n$
(see the proof of Theorem~\ref{thm:HBn-}).

\subsection{Split classes for $\Gamma_n$}

We first realize the Hecke-Clifford algebra $\aHC_{B_n}$ as a spin
group algebra for a finite group $\Gamma_n$. Define the semidirect
product
\begin{align*}
\Gamma_n :&= \mathbb Z_2^n \rtimes B_n,
\end{align*}
which as a group is isomorphic to a wreath product $(\mathbb Z_2
\times \mathbb Z_2)^n \rtimes S_n$. So it is well known (see
\cite[I, Appendix~B]{Mac}) that its conjugacy classes $\mathcal
C_\rho$ are parametrized by quadruples of partitions
$\rho=(\rho_{++}, \rho_{+-}, \rho_{-+}, \rho_{--})$ of total size
$n$, where the second signs in the subscripts are understood to
correspond to the second factor in $\Z_2\times \Z_2$ (which has its
origin from $B_n$). Denote by $\ell (\rho) =\ell (\rho_{++})+ \ell
(\rho_{+-})+ \ell (\rho_{-+})+ \ell (\rho_{--}).$

Consider a finite group
$$
\Pi_n = \langle a_1, \dots, a_n, z | a_i^2 = z^2 = 1,  a_iz = za_i, a_ia_j =
za_ja_i \;(i \ne j)\rangle,
$$
and write the generators of $B_n =\Z_2^n\rtimes S_n$ as $\tau_1,
\dots, \tau_n, s_1, \dots, s_{n-1}$, where $\tau_i$ is a generator
of the $i$th copy of $\Z_2$. Then the semidirect product $\wtd
\Gamma_n = \Pi_n \rtimes B_n$ is a group such that $z$ is central,
$a_j s_i=s_i a_{s_i(j)}$, and
$$
\tau_i a_j =
\begin{cases}
za_j\tau_i & \mbox{ if }  i = j
 \\
a_j\tau_i & \mbox{ if }  i \ne j.
\end{cases}
$$
The group $\wtd \Gamma_n$ is a double cover of $\Gamma_n$:
$$
1 \longrightarrow \{1, z\} \longrightarrow \wtd\Gamma_n
\stackrel{\theta}{\longrightarrow} \Gamma_n \longrightarrow 1.
$$
Introduce the spin group algebra $ \C\Gamma_n^- := \C \wtd \Gamma_n
/\langle z+1\rangle$. The quotient algebra $\C \Pi_n/\langle
z+1\rangle$ is identified with $\Cl_n$ by $\bar a_i \mapsto c_i$,
which leads to a natural identification of the superalgebras
\begin{equation} \label{eq:id:hcb}
\C\Gamma_n^- = \aHC_{B_n},
\end{equation}
where the superalgebra structure on $\C \Gamma_n^-$ is given by
letting each $\bar a_i$ be odd and each $s_i$ and $\tau_i$ be even.
We feel free to use the identification \eqref{eq:id:hcb} below:
while $\aHC_{B_n}$ appears to be super-equivalent to $\C B_n^-$,
$\C\Gamma_n^-$ allows one to appeal to finite group techniques.

For an
(ordered) subset $I=\{i_1,\ldots,i_r\}$, we denote
$a_I=a_{i_1}\ldots a_{i_r}$, and similarly for $\tau_I$. An
arbitrary element $z^*a_I\tau_J\sigma \in \wtd \Gamma_n$ may be
written as a product
$$
z^*a_I\tau_J\sigma =
z^*(a_{I_1}\tau_{J_1}\sigma_1)\cdots(a_{I_k}\tau_{J_k}\sigma_k),
$$
where $*\in \{0, 1\}$, $\sigma = \sigma_1\cdots \sigma_k \in S_n$ is
a product of disjoint cycles, and $I_a, J_a \subseteq
\supp(\sigma_a)$ for all $1\le a\le k$. Note that $|I|=\sum_{i=1}^k
|I_i|.$

Let $\mathcal C_\rho$ be a split conjugacy class of $\Gamma_n$. Then
its inverse image in $\wtd \Gamma_n$ is $\theta^{-1}(\mathcal
C_\rho) = \mathcal C_\rho^+ \sqcup z\mathcal C_\rho^+.$ In
particular, we can make sense of split classes in $\Gamma_n$ and
$\wtd \Gamma_n$ as before.

\begin{prop}  \label{prop:ccGamman}
Let $\mathcal C_\rho$ be a conjugacy class of $\Gamma_n$.  Then
$\mathcal C_\rho$ is even split if and only if   $\rho \in (\OP,
\EP, \emptyset, \emptyset)$.
\end{prop}

\begin{proof}
$(\Rightarrow)$   Let $\mathcal C_\rho$ be an even conjugacy class
of $\Gamma_n$.

Case 1: Suppose $\rho_{++} \notin \OP$.  Then $\rho_{++}$ has at
least one even part, so $\theta^{-1}(\mathcal C_\rho)$ contains a
product of disjoint cycles of the form $a_I\tau_J\sigma=(1,\ldots,
r)(a_{I_2}\tau_{J_2}\sigma_2) \cdots (a_{I_l}\tau_{J_l}\sigma_l)$,
where $r$ is even; note $|I|$ is also even, since $\mathcal C_\rho$
is even.  We compute the following conjugation of $a_I\tau_J\sigma$:
 {\allowdisplaybreaks
\begin{align*}
a_{1\dots r}^{-1}  (a_I\tau_J\sigma) a_{1\dots r}
  & =  a_{r\dots 1} (1,\ldots, r) a_{1\dots r}
 z^{|I|r} (a_{I_2}\tau_{J_2}\sigma_2) \cdots (a_{I_l}\tau_{J_l}\sigma_l)
    \\
& =  a_{r\dots 1} a_{2\dots r}a_1 (1,\ldots, r)
(a_{I_2}\tau_{J_2}\sigma_2) \cdots (a_{I_l}\tau_{J_l}\sigma_l)
    \\
& =  z a_I\tau_J\sigma.
\end{align*}}
Thus $\mathcal C_\rho$ does not split if $\rho_{++} \notin \OP$.  \\

Case 2: Suppose $\rho_{+-}\notin \EP$.  Then $\rho_{+-}$ has an odd
part, so $\theta^{-1}(\mathcal C_\rho)$ contains an element of the
form $a_I\tau_J\sigma=(\tau_1(1,\ldots, r) )
(a_{I_2}\tau_{J_2}\sigma_2) \cdots (a_{I_l}\tau_{J_l}\sigma_l)$,
where $r$ is odd; note  $|I|=\sum_i |I_i|$ is even, since $\mathcal
C_\rho$ is even. We compute the following conjugation:
{\allowdisplaybreaks
\begin{align*}
a_{1\dots r}^{-1} (a_I\tau_J\sigma) a_{1\dots r}
   =  a_{r\dots 1}\tau_1(1,\ldots, r) a_{1\dots r}z^{|I|r} (a_{I_2}
\tau_{J_2}\sigma_2) \cdots (a_{I_l}\tau_{J_l}\sigma_l)
 =  z a_I\tau_J\sigma,
\end{align*}}
where we used $a_{r\dots 1}\tau_1(1,\ldots, r)a_{1\dots
r}=z\tau_1a_{r\dots 1} a_{2\dots r}a_1 (1,\ldots, r) =  z \tau_1
(1,\ldots, r).$

Thus $\mathcal C_\rho$ does not split if $\rho_{+-}\notin \EP$.  \\

Case 3:  Suppose $\rho_{-+} \neq \emptyset$.  Then
$\theta^{-1}(\mathcal C_\rho)$ contains an element $a_I\tau_J\sigma$
of the form $(a_1(1,\ldots, r)) (a_{I_2}\tau_{J_2}\sigma_2) \cdots
(a_{I_l}\tau_{J_l}\sigma_l)$.  Note $|I|$ is even, since $\mathcal C_\rho$ is even. We compute the
following conjugation:
 {\allowdisplaybreaks
\begin{align*}
(a_1(1, \ldots, & r))^{-1} (a_I\tau_J\sigma) (a_1 (1,\ldots, r)) \\
& = (a_{I_2}\tau_{J_2}\sigma_2) \cdots (a_{I_l}\tau_{J_l}\sigma_l)
(a_1 (1,\ldots, r))
    \\
& =  z^{|I|-1} a_1 (1,\ldots, r)(a_{I_2}\tau_{J_2}\sigma_2) \cdots
(a_{I_l}\tau_{J_l}\sigma_l)
    \\
& =  z a_I\tau_J\sigma.
\end{align*}}
Thus $\mathcal C_\rho$ does not split if $\rho_{-+} \neq \emptyset$.\\

Case 4:  Suppose $\rho_{--} \neq \emptyset$.  Then
$\theta^{-1}(\mathcal C_\rho)$ contains an element of the form
$a_I\tau_J\sigma=(a_1\tau_j (1\cdots r)) (a_{I_2}\tau_{J_2}\sigma_2)
\cdots (a_{I_l}\tau_{J_l}\sigma_l)$, where $j \in \{1, \dots , r\}$.
Note $|I|$ is even, since $\mathcal C_\rho$ is even. We compute the following conjugation:
{\allowdisplaybreaks
\begin{align*}
(a_1\tau_j(1, & \ldots, r))^{-1} (a_I \tau_J\sigma)
(a_1\tau_j(1,\ldots, r))
    \\
& = (a_{I_2}\tau_{J_2}\sigma_2) \cdots (a_{I_l}\tau_{J_l}\sigma_l)
(a_1\tau_j(1,\ldots, r))
    \\
& =  z^{|I|-1}a_1\tau_j(1,\ldots, r) (a_{I_2}\tau_{J_2}\sigma_2)
\cdots (a_{I_l}\tau_{J_l}\sigma_l)
    \\
& =  z a_I \tau_J\sigma.
\end{align*}}
Thus $\mathcal C_\rho$ does not split if $\rho_{--} \neq \emptyset$.

Hence, we have shown that if $\mathcal C_\rho$ is even split then
$\rho \in (\OP, \EP, \emptyset, \emptyset)$.

($\Leftarrow$) Suppose $\rho \in (\OP, \EP, \emptyset, \emptyset)$.
Then the conjugacy class $\mathcal C_\rho$ is clearly even.

Suppose $\mathcal C_\rho$  does not split, i.e.~any $x\in \mathcal
C_\rho$ is conjugate to $zx$. Take an element in
$\theta^{-1}(\mathcal C_\rho)$ of the form $\tau_K
\sigma=\sigma_1\dots \sigma_p (\tau_{k_1} \sigma'_1)\cdots
(\tau_{k_q}\sigma'_q)$, with each $\sigma_i$ an odd cycle, each
$\sigma'_i$ an even cycle, $p = \ell (\rho_{++}), q = \ell
(\rho_{+-})$, $k_i \in \supp(\sigma'_i)$, and $\sigma =
\sigma_1\dots \sigma_p \sigma'_1 \dots \sigma'_q$ a product of
disjoint cycles.

Since $\mathcal C_\rho$ does not split, there exists $a_Jt \in \Pi_n
\rtimes B_n$ such that $a_Jt \tau_K\sigma = z\tau_K\sigma a_J t$.
Then we must have $z a_J = a_{\tau_K\sigma(J)}$,
$\supp(\tau_K\sigma) \subseteq J$, and $t \tau_K\sigma =
\tau_K\sigma t$. On the other hand, we compute
\begin{align*} 
a_{\tau_K\sigma(J)} & =  \tau_K\sigma a_J (\tau_K\sigma)^{-1}
    \\
& =  \sigma_1\dots \sigma_p (\tau_{k_1} \sigma'_1) \cdots
(\tau_{k_q}\sigma'_q) a_J (\tau_K \sigma)^{-1}
    \\
& =  z^{q+|\rho_{++}| + |\rho_{+-}| -\ell(\rho_{++})
-\ell(\rho_{+-})}a_J
    \\
& =  z^{|\rho_{++}|-\ell(\rho_{++})} a_J
     =  a_J
\end{align*}
which is a contradiction to  $z a_J =a_{\tau_K\sigma(J)}$. So
$\mathcal C_\rho$ must split.
\end{proof}

For $\rho\in (\OP, \EP, \emptyset, \emptyset)$, the split class
$\theta^{-1}(\mathcal C_\rho)$ is a disjoint union of two conjugacy
classes; if we denote the class containing an element of $B_n$ by
$\mc C_\rho^+$, then the other is $z\mc C_\rho^+$, and
$\theta^{-1}(\mathcal C_\rho)=\mc C_\rho^+\sqcup z\mc C_\rho^+$.

\subsection{Simple modules of $\aHC_{B_n}$}

Propositions~\ref{functors} and \ref{spinBnirreps} imply that the
simple $\aHC_{B_n}$-modules are parametrized by $\la \vdash n$, and
that they are all of type $\typeM$. We shall construct them
explicitly, and then match them with the $\C B_n^-$-modules $B^\la$
via the super-equivalence in Proposition~\ref{functors}.

We adopt the convention that $c_{-i} =- c_i$ for $1\le i\le n$.  The
algebra $\aHC_{B_n}$ acts on the Clifford algebra $\Cl_n$ by the
formulas
$$
c_i.(c_{i_1}c_{i_2}\ldots) =c_ic_{i_1}c_{i_2}\ldots, \quad
\sigma.(c_{i_1}c_{i_2}\ldots)=c_{\sigma(i_1)}c_{\sigma(i_2)} \ldots,
$$
for $\sigma \in B_n$ and all $i$. This $\aHC_{B_n}$-module $\Cl_n$
is called the {\em basic spin module}.

\begin{lem}   \label{lem:basicSpinBn}
The character value of the  basic spin $\aHC_{B_n}$-module at an
even split conjugacy class $\mc{C}_{\rho}^+$ is equal to
$2^{\ell(\rho)}$ for $\rho\in (\OP, \EP, \emptyset, \emptyset)$,
and $0$ elsewhere.
\end{lem}

\begin{proof}
Let $\sigma=\sigma_1\ldots \sigma_\ell \in B_n$ be a product of
disjoint signed cycles of type $\rho$. The elements $c_I :=\prod_{i
\in I} c_i$ (which are defined up to a nonessential sign) for
$I\subset \{1, \ldots, n\}$ form a basis of the  basic spin module
$\Cl_n$. Observe that $ \sigma c_I =  c_I$ if $I$ is a union of a
subset of the supports $\text{supp} (\sigma_p)$ for $1\le p \le
\ell(\alpha)$; otherwise $\sigma c_I$ is equal to $\pm c_J$ for some
$J \neq I$. Hence the lemma follows.
\end{proof}

Let $\la \vdash n$. Via pullback of the canonical projection $B_n
=\Z_2^n\rtimes S_n  \rightarrow S_n$, the Specht module $S^\la$ is
endowed with a $B_n$-module structure. Then
\begin{equation}   \label{eq:Kla}
K^\la := \Cl_n \otimes S^\la
\end{equation}
is naturally a module over $\aHC_{B_n}=\Cl_n\rtimes B_n$ (and hence
also a   module of the group $\wtd \Gamma_n$), where $\Cl_n$ acts by
left multiplication on the first tensor factor and $B_n$ acts
diagonally.

Denote by $\varphi^\la$ the character of the module $K^\la$ (of the
group $\wtd \Gamma_n$). Note that the character value
$\varphi^\la(x)$ is zero unless $x$ is even split. There is a
canonical bijection between the types $\rho
=(\alpha,\beta,\emptyset,\emptyset)$ in $(\OP, \EP, \emptyset,
\emptyset)$ such that $|\alpha|+|\beta|=n$ and partitions of $n$ (by
taking $\alpha\cup\beta$), and 
we shall denote the resulting %
partition 
 $\rho$ again by abuse of notation. %
Note also that $x\in
\mc{C}_\rho^+$ implies that $x^{-1} \in \mc{C}_\rho^+$. By
Lemma~\ref{lem:basicSpinBn} and the definition~\eqref{eq:Kla}, we
conclude that
\begin{equation} \label{eq:chvalue}
\text{ the character value $\varphi^\la(x)$  at $x\in \mc{C}_\rho^+$
is $2^{\ell(\rho)}\chi^\la_{\rho}$,
 }
 \end{equation}
where we recall $\chi^\la_\rho$ denotes the character value of
$S^\la$ at an element in $S_n$ of cycle type $\rho$.

Thanks to the isomorphism $\Phi : \C \Gamma_n^- = \aHC_{B_n} =\Cl_V
\rtimes B_n \rightarrow \Cl_\h \otimes \C B_n^-$ from
Proposition~\ref{isofinite} for $W=B_n$, the Morita
super-equivalence in Proposition~\ref{functors} applies.  Further
computation is necessary to determine which simple $\C B_n^-$-module
corresponds under the super-equivalence to the simple
$\aHC_{B_n}$-module $K^\la$.

\begin{lem}  \label{lem:K=B}
The $\aHC_{B_n}$-module $K^\la$ corresponds to the $\C B_n^-$-module
$B^\la$ under the bijection induced by $\mf G$ in
Proposition~\ref{functors}.
\end{lem}
\begin{proof}
We shall compute the characters of the modules $U \otimes B^\la$ of
$\aHC_{B_n}$ (and hence of $\wtd \Gamma_n$) on an arbitrary even
split class. A canonical representative for an even split class of
$\Gamma_n$ of type $(\rho_{++}, \rho_{+-},  \emptyset, \emptyset)
\in (\OP, \EP, \emptyset, \emptyset)$ can always be chosen inside
the subgroup $B_n$ as follows. We choose the canonical positive
$(++)$-cycle inside $B_n$ which permutes $a$ through $a+k$ to be
$s_a\cdots s_{a+k-1}$, and the canonical negative $(+-)$-cycle
inside $B_n$ which permutes those same elements to be $s_a\cdots
s_{a+k-1}\tau_{a+k}$. A representative element for the conjugacy
class of type $(\alpha, \beta, \emptyset, \emptyset)$ is chosen to
be a product of such disjoint canonical cycles of the appropriate lengths, so that the supports
of negative cycles are bigger than those of positive cycles.

Let $b_J \sigma$ be such a canonical representative element of the
conjugacy class of type $(\rho_{++}, \rho_{+-},  \emptyset,
\emptyset)$, and consider the images of its component cycles under
$\Phi$.  Since there will be no cancellation between cycles, by
Proposition~ \ref{charU} we may ignore terms containing nontrivial
products of $c_i$.

Consider an odd positive $(k+1)$-cycle $s_a \cdots s_{a+k-1}$ in
$b_J \sigma$, for $k$ even. Similar to the proof of
Proposition~\ref{prop:charB}, we compute
\begin{align}
\Phi(s_a\cdots s_{a+k-1})&= (-1)^{k+\frac k2} \beta_a t_a \cdots \beta_{a+k-1}t_{a+k-1}
 \notag \\
& = (-1)^{k+\frac k2 +\frac{k(k-1)}2} \beta_a \cdots \beta_{a+k-1} t_a \cdots t_{a+k-1}
 \label{eq:stbeta} \\
& = 2^{-\frac k2}(-1)^{2k + \frac{k(k-1)}2} t_a \cdots t_{a+k-1} +
(\mbox{terms involving }c_i)
 \notag \\
&  = 2^{-\frac k2}(-1)^{\frac k2} t_a \cdots t_{a+k-1} +
(\mbox{terms involving }c_i).
 \notag
\end{align}

Now consider an even negative $(k+1)$-cycle $s_a \cdots
s_{a+k-1}\tau_{a+k}$ in $b_J \sigma$, for $k$ odd. Recall $\tau_i$
is the generator of the $i$th copy of $\Z_2$ inside $B_n$, and $b_i$
is defined in \eqref{eq:bi} for $1\le i\le n$. By
\cite[Lemma~5.4]{KW2}, we have
\begin{equation}  \label{eq:KW2}
\Phi (\tau_i) = (-1)^{n-i-\frac12} c_i b_i, \quad \text{ for } 1\le
i\le n.
\end{equation}
As the positive cycles have supports in terms of smaller numbers
than the negative cycles, we must have $n-a \equiv 1$ mod 2. Using
this parity condition, Table~B for type $B_n$, \eqref{eq:stbeta} and
\eqref{eq:KW2}, we compute
\begin{align*}
\Phi(s_a\cdots& s_{a+k-1}\tau_{a+k})
  \\
& =(-1)^{k + \frac k2 + \frac{k(k-1)}2+ n-a-k - \frac 12} \beta_a
\cdots \beta_{a+k-1} t_a\cdots t_{a+k-1} c_{a+k}b_{a+k}
 \\
&=  2^{-\frac k2}(-1)^{\frac{k(k-1)}2+ n-a} t_a\cdots
t_{a+k-1}b_{a+k} + (\mbox{terms involving }c_i)
 \\
&=  2^{-\frac k2}(-1)^{\frac{k^2-k}2+1}t_a \cdots t_{a+k-1}b_{a+k} +
(\mbox{terms involving }c_i)
 \\
&= 2^{-\frac k2}(-1)^{\frac{k+1}2} t_a\cdots t_{a+k-1}b_{a+k} +
(\mbox{terms involving }c_i).
\end{align*}
Here the last identity follows since $\frac{k^2-k}2+1 \equiv
\frac{k+1}2$ mod 2 whenever $k$ is odd.

Multiplying the images of canonical cycles together, we obtain
$$
\Phi(b_J \sigma) = 2^{-\frac{n-\ell(\rho_{++}\cup\rho_{+-})}2}
(-1)^{\frac{n-\ell(\rho_{++})}2}\sigma + (\mbox{terms involving
}c_i).
$$

So the character value of $U \otimes B^\la$ on $b_J \sigma$ is
\begin{equation*}
\begin{cases}
2^{\ell(\rho_{++}\cup\rho_{+-})} \chi^\la_{\rho_{++}\cup\rho_{+-}} & \mbox{ if $n$ is even}\\
2^{\ell(\rho_{++}\cup\rho_{+-})+1} \chi^\la_{\rho_{++}\cup\rho_{+-}}
& \mbox{ if $n$ is odd,}
\end{cases}
\end{equation*}
which is equal to the character value of $K^\la$ on $\Phi(b_J
\sigma)$ when $n$ is even, and twice that (since $B^\la$ is of type
$\typeQ$) when $n$ is odd; see \eqref{eq:chvalue}.
\end{proof}

\begin{prop}  \label{prop:Kla}
$\{K^\la| \la \vdash n\}$ is a complete list of pairwise
inequivalent simple $\aHC_{B_n}$-modules, all of type $\typeM$.
\end{prop}

\begin{proof}
By Proposition~\ref{spinBnirreps} and Lemma~\ref{lem:K=B}, $\{K^\la
\mid \la \vdash n\}$ are pairwise inequivalent simple
$\aHC_{B_n}$-modules, all of type $\typeM$. As we already know by
Proposition~\ref{prop:ccGamman} that the total number of simple
modules is the number of partitions of $n$, the proof is completed.
\end{proof}

\subsection{The characteristic map for $\aHC_{B_n}$}

Let $R^-(\Gamma_n)$ be the Grothendieck group of the category of
$\C\Gamma_n^-$-modules, which can also be identified with the free
abelian group with a basis made up of the irreducible
$\C\Gamma_n^-$-characters. Define
$$
\breve{R} = \bigoplus_{n=0}^\infty R^-(\Gamma_n),
$$
where $R^-(\Gamma_0) = \mathbb Z$.

We shall define a ring structure on $\breve{R}$ as follows. For a
$\C \Gamma_m^-$-module $M$ and a $\C\Gamma_n^-$-module $N$, we
define the product
$$
[M] \cdot [N] = [\C \Gamma_{m+n}^- \otimes_{\C \Gamma_m^- \times \C
\Gamma_n^-} (M \otimes N)],
$$
and then extend by $\Z$-bilinearity. It follows from the properties
of the induced characters that the multiplication on $\breve{R}$ is
commutative and associative. Given $\C \Gamma_n^-$-modules $M,N$, we
define a bilinear form on $\breve{R}$ by letting
\begin{equation}  \label{eq:bilformR}
\langle M,N \rangle  =\dim \Hom_{\C \Gamma_n^-} (M,N).
\end{equation}
Now define the characteristic map $\ch : \breve{R} \rightarrow \La$
as the linear map
$$
\ch(\phi) = \sum_{\mu \vdash n} z^{-1}_\mu 2^{-\ell(\mu)} \phi_\mu
p_\mu, \quad \text{ for } \phi \in R^-(\Gamma_n).
$$

\begin{lem}
\label{lem:phipsichar} Let $\phi \in R^-(\Gamma_m), \psi \in
R^-(\Gamma_n)$, and $\gamma \in (\OP, \EP, \emptyset, \emptyset)$.
Then
$$
(\phi\cdot\psi)(\gamma) = \sum_{\alpha, \beta}
\frac{z_\gamma}{z_\alpha z_\beta} \phi(\alpha) \psi(\beta)
$$
where the sum is taken over $\alpha, \beta \in (\OP, \EP, \emptyset,
\emptyset)$ such that $\gamma =\alpha \cup \beta$.
\end{lem}

\begin{proof}
Let $g \in \wtd\Gamma_{m+n}$ be an element of type $\gamma$. By
\cite[I, Appendix~B, (3.1)]{Mac}, the order of the centralizer in
$\Gamma_n$ of an element of type $\rho
=(\alpha,\beta,\emptyset,\emptyset) \in (\OP, \EP, \emptyset,
\emptyset)$ is
%
$z_\alpha z_\beta 4^{\ell(\alpha\cup\beta)}.$
Now we compute
  {\allowdisplaybreaks
\begin{align*}
(\phi \cdot \psi)(\gamma)
 &= \frac 1 {\vert \wtd{\Gamma}_{m,n}\vert}
\sum_{h \in \wtd \Gamma_{m+n}} (\phi  \times \psi) (h^{-1}gh)
    \\
& = \frac{| \wtd{\Gamma}_{n+m}|}{|\wtd{\Gamma}_{m,n}||\mathcal
C_\gamma^+|} \sum_{w \in \mathcal C_\gamma^+} (\phi \times \psi) (w)
    \\
& =  \frac{2^{2\ell(\gamma)}z_\gamma}{m!n!2^{2m+2n}} \sum_{\alpha
\cup \beta = \gamma} \phi(\alpha)\psi(\beta)\vert\mathcal
C_\alpha^+\vert\vert\mathcal C_\beta^+\vert
    \\
& =  \frac{2^{2\ell(\gamma)}z_\gamma}{m!n!2^{2m+2n}} \sum_{\alpha
\cup \beta = \gamma} \phi(\alpha)\psi(\beta) \frac{2^{2m} m!}{z_\alpha
2^{2\ell(\alpha)}} \frac{2^{2n}n!}{z_\beta 2^{2\ell(\beta)}}
    \\
& =  \sum_{\alpha, \beta} \frac{z_\gamma}{z_\alpha z_\beta}
\phi(\alpha) \psi(\beta).
\end{align*}}
The lemma is proved.
\end{proof}

\begin{thm} \label{ch}
The characteristic map $\ch: \breve{R} \rightarrow \La$ is an
isometric isomorphism of graded algebras, which sends $[K^\la]$ to
$s_\la$ for all $\la$.
\end{thm}

\begin{proof}
Recall that the character $\varphi^\la$ of the irreducible module
$K^\la$ is $2^{\ell(\alpha\cup\beta)}\chi^\la_{\alpha\cup\beta}$ on
elements of type $(\alpha, \beta,\emptyset,\emptyset)\in(\OP,\EP,
\emptyset,\emptyset)$, and 0 otherwise.  Thus
$$
\ch(\varphi^\la) = \sum_{(\alpha,\beta) \in
(\OP,\EP)}z_\alpha^{-1}z_\beta^{-1} 2^{-\ell(\alpha\cup\beta)}
2^{\ell(\alpha\cup\beta)}\chi^\la_{\alpha\cup\beta}p_{\alpha\cup\beta}
= s_\la.
$$
This map sends an orthonormal basis for $\breve{R}$ to an
orthonormal basis of $\La$, so it is an isometry.

Now we compute the image of a product under the characteristic map.
Let $\phi, \psi$ be as in the previous lemma.  Then
\begin{align*}
\ch(\phi\cdot\psi) &= \sum_{\gamma\vdash m+n}
z_\gamma^{-1}2^{\ell(\gamma)}(\phi\cdot\psi)(\gamma) p_\gamma
 \\
= \sum_\gamma& \sum_{\alpha, \beta: \alpha\cup\beta=\gamma}
z_\gamma^{-1} \frac{z_\gamma}{z_\alpha z_\beta} \phi(\alpha)
\psi(\beta) 2^{\ell(\gamma)}p_\gamma =
\ch(\phi)\ch(\psi),
\end{align*}
so $\ch$ is also an algebra isomorphism.
\end{proof}

\subsection{Spin fake degrees for $\aHC_{B_n}$}

Let $x, y$, and $z$ be three (possibly infinite) sets of independent
indeterminates. The super Schur functions and the super Cauchy identity are
standard, and we refer to \cite[Section 5.3]{WW2} for more details.
For a partition $\la$, the {\em super Schur function} $hs_\la$ is defined to be
$$
hs_\la(x; y) := \sum_{\mu \subseteq \la} s_\mu(x)s_{\la'/\mu'}(y).
$$
We have the following super Cauchy identity:
\begin{equation}\label{sCi}
\frac{\prod_{j,k}(1+y_jz_k)}{\prod_{i,k}(1-x_iz_k)} = \sum_{\la \in
\calP} hs_\la(x;y)s_\la(z).
\end{equation}

Let $a,b$ be indeterminates. The formula $(*)$ in \cite[I.3.3]{Mac}
was interpreted in \cite[(5.13)]{WW2} as a specialization of $hs_\la
(x;y)$, by letting $x = aq^\bullet =(a, aq, aq^2,\ldots)$ and $y =
bq^\bullet$:
\begin{equation}\label{hslaspec}
hs_\la(aq^\bullet; bq^\bullet) = q^{n(\la)} \prod_{\square \in
\la}\frac{a+bq^{c_\square}}{1-q^{h_\square}}.
\end{equation}

The following theorem was used in the proof of
Theorem~\ref{thm:HBn-} earlier.

\begin{thm}\label{thm: sfd HBn}
The graded multiplicity of the irreducible $\mathfrak
H_{B_n}^c$-module  $K^\lambda$ in the $\mathfrak
H_{B_n}^c$-module $\Cl_\h \otimes S^*V$ is
$$
H_{B_n}(\la,t) = hs_\la(t^{2\bullet}; t^{2\bullet + 1}) =
t^{2n(\lambda)}\prod_{\square\in
\lambda}\frac{1+t^{2c_\square+1}}{1-t^{2h_\square}}.
$$
\end{thm}

\begin{proof}
The second identity follows by \eqref{hslaspec} with $a=1, b=t$ and
$q=t^2$. So it remains to prove the first identity.

Let us compute the image of $\Cl_V \otimes S_tV$ under the
characteristic map. The character of the basic spin
$\C\Gamma_n^-$-module $\Cl_V$ on any representative in the conjugacy
class of type $(\alpha, \beta, \emptyset, \emptyset) \in (\OP, \EP,
\emptyset, \emptyset)$ is $2^{\ell(\alpha\cup\beta)}$; we shall
choose the representative to be the canonical element $b_J\sigma \in
B_n$ (with $\sigma\in S_n$) of type $(\alpha, \beta) \in (\OP, \EP)$
as in the proof of Lemma~\ref{lem:K=B}.

Now we compute the character value of $b_J\sigma $ on the
$B_n$-module $S^*V$.  We shall denote by $\ell_1 =\ell(\alpha)$,
$\ell_2 =\ell(\beta)$, $\ell =\ell(\alpha) +\ell(\beta)$. We write
$$
\sigma = (1,\dots,\alpha_1)(\alpha_1+1, \dots,
\alpha_1+\alpha_2)\cdots(|\alpha|+1, \dots,
|\alpha|+\beta_1)\cdots(n-\beta_{\ell_2}+1,\ldots, n).
$$
Thus $\sigma$ will permute the monomial
basis of $S^*V$, fixing only those monomials
$$
\underline{x}^{\underline{a}} := (x_1x_2\cdots x_{\alpha_1})^{a_1}
(x_{\alpha_1+1}\ldots x_{\alpha_1+\alpha_2})^{a_2}
\cdots(x_{n-\beta_{\ell_2}+1}\cdots x_n)^{a_{\ell}}
$$
for nonnegative integers $a_1, \dots, a_{\ell}$. Note that
$$
b_J\sigma(\underline{x}^{\underline{a}}) =(-1)^{a_{\ell_1+1}+\ldots
+a_{\ell}} \underline{x}^{\underline{a}}.
$$
This implies that the character value of $S_tV$ on $b_J\sigma$ is
\begin{align*}
\tr (b_J\sigma) \vert_{S_tV}
 &=\sum_{a_1,\ldots, a_\ell \ge 0} (-1)^{a_{\ell_1+1}+\ldots
+a_{\ell}} t^{\sum_i a_i\alpha_i +\sum_j a_{\ell_1+j}\beta_j}
  \\
& =\frac 1 {(1-t^{\alpha_1}) \cdots
(1-t^{\alpha_{\ell_1}})
(1+t^{\beta_1}) \cdots
(1+t^{\beta_{\ell_2}})}
  \\
& = \frac 1 {(1+ (-t)^{\rho_1}) \cdots (1+(-t)^{\rho_\ell})},
\end{align*}
where we have switched notation in the last equation using
$\rho=\alpha\cup\beta= (\rho_1,\rho_2,\ldots, \rho_\ell)$. Then the
character value of $\Cl_V \otimes S_tV$ on $b_J\sigma$ is
\begin{equation}  \label{eq:trbsigma}
\tr (b_J\sigma) \vert_{\Cl_V\otimes S_tV} = \frac
{2^{\ell(\rho)}}{(1+ (-t)^{\rho_1}) \cdots (1+(-t)^{\rho_\ell})}.
\end{equation}

Given a power series $f(u)$ in a variable $u$, we denote by
$[u^n]f(u)$ the coefficient of $u^n$ in the series expansion of
$f(u)$. Applying the characteristic map to $\Cl_n\otimes S_tV$ with
the help of \eqref{eq:trbsigma}, we compute
 {\allowdisplaybreaks
\begin{eqnarray*}
\ch(\Cl_n\otimes S_tV) & = & \sum_{\rho \vdash n} z_\rho^{-1}
2^{-\ell(\rho)} \frac  {2^{\ell(\rho)}}{(1+ (-t)^{\rho_1}) \cdots
(1+(-t)^{\rho_\ell})} p_\rho
 \\
& = & \left[u^n\right] \sum_{\rho\in \calP} z_\rho^{-1} p_\rho \frac
{u^{|\rho|}}{(1+ (-t)^{\rho_1}) \cdots (1+(-t)^{\rho_\ell})}
 \\
 & = & \left[u^n\right]  \prod_{i,j} \left(\frac 1{1-x_iu(-t)^j}\right)^{(-1)^j}
 \\
& = & \sum_{\lambda\vdash n} hs_\lambda (t^{2\bullet}; t^{2\bullet +
1})s_\lambda(x).
\end{eqnarray*}}
The last equation used the super Cauchy identity \eqref{sCi}. On the
other hand, since each $K^\la$ is simple of type $\typeM$ by
Proposition~\ref{prop:Kla}, we have
$$
\ch(\Cl_n\otimes S_tV) =\sum_{\lambda\vdash n} H_{B_n}(\la, t)\
s_\lambda(x).
$$
The first identity in the theorem now follows from comparing the
above two expressions for $\ch(\Cl_n\otimes S_tV),$ and the linear
independence of $s_\la$'s.
\end{proof}

The following is equivalent to Theorem~\ref{thm: sfd HBn} by
Lemma~\ref{lem:P=W} and using the fact that the degrees of $B_n$ are
$2,4, \ldots, 2n$.

\begin{thm}  \label{sfd:HCBn}
The spin fake degree of the irreducible $\mathfrak H_{B_n}^c$-module
$K^\lambda$ is
$$
P_{B_n}(\la,t) = t^{2n(\lambda)}\prod_{\square\in
\lambda}\frac{1+t^{2c_\square+1}}{1-t^{2h_\square}}(1-t^2)(1-t^4)\cdots(1-t^{2n}).
$$
\end{thm}

\section{The spin fake degrees of type $D_n$ (for $n$ odd)}
\label{sec:Dn:odd}

Let $n$ be odd throughout this section.
\subsection{Structure of the algebra $\C D_n^-$ for $n$ odd}

Recall from Table~A the definition of $\mathbb C D_n^-$ via
generators $t_1,\ldots, t_n$ and defining relations. Denote by
$\Cl_{n}^0 \subset \Cl_n$ the even subalgebra of the Clifford
superalgebra $\Cl_n$. The algebra $\Cl_n^0$ is abstractly a Clifford
algebra in $n-1$ generators $c_ic_n$ for $1\le i\le n-1$. Define
$$
\zeta := (-1)^{\frac{n(n-1)}2}c_1c_2 \cdots c_n \in \Cl_n.
$$
Note that $\zeta \not\in \Cl_n^0$.

\begin{lem}  \label{lem:zeta}
For $n$ odd, the element $\zeta$ commutes with every $c_i$.
Moreover, $\zeta^2 = 1$.
\end{lem}

\begin{proof}
Follows by a direct computation.
\end{proof}

The tensor algebra $\Cl_{n}^0 \otimes \C S_n$ carries a superalgebra
structure by letting each $c_ic_n$ be even and each $s_i$ odd, for
$1\le i\le n-1$. In particular, $\Cl_{n}^0$ is a purely even
algebra, and it follows by \eqref{eq:tensorsuper} that $c_ic_n$
commutes with $s_j$ for all possible $i,j$.

\begin{thm}\label{Dnodd iso}
For $n$ odd, there exists an isomorphism of superalgebras
\begin{align*}
\phi^D & :  \; \mathbb CD_n^-
\stackrel{\cong}{\longrightarrow}  \Cl_{n}^0 \otimes \C
S_n,
 \\
t_i \mapsto &
 \begin{cases}
 \frac 1{\sqrt 2}\zeta (c_i - c_{i+1}) s_i, & \quad
\mbox{ for } 1\le i\le n-1,
 \\
\frac1{\sqrt 2}\zeta(c_{n-1}+c_n)s_{n-1}, & \quad \mbox{ for } i=n.
\end{cases}
\end{align*}
(We emphasize here that each $s_i$ is odd.)
\end{thm}

\begin{proof}
First note that $\phi^D (t_i)$ for each $i$ is indeed in $\Cl_{n}^0
\otimes \C S_n$, since $n$ is odd.

Recall from Theorem~\ref{CBn-iso} the superalgebra isomorphism
$\phi^B: \C B_n^- \stackrel{\cong}{\rightarrow} \Cl_n \otimes \C
S_n.$ The images of $t_i$ for $1\le i\le n-1$ under $\phi^D$ and
$\phi^B$ differ exactly by a factor $\zeta$. All the relations for
$\C D_n^-$ which do not involve $t_n$ in Table~B are identical for
types $B$ and $D$ and they all involve even numbers of these
$t_i$'s, hence they are preserved by $\phi^D$ because $\zeta^2=1$
and $\phi^B$ is a homomorphism. In addition, it is straightforward
to check by definition and Lemma~\ref{lem:zeta} the remaining
relations:
\begin{align*}
\big( \phi^D(t_i)\phi^D(t_{n})\big)^2=-1, &\quad \text{ for } i\neq
n-2, n,
 \\
\phi^D(t_{n})^2=1, \qquad & \big( \phi^D(t_n)\phi^D(t_{n-2})\big)^3
=1.
\end{align*}
For example, we compute
$$
\big( \phi^D(t_n)\phi^D(t_{n-2})\big)^3 = -\frac 18 \zeta^6
((c_{n-1}+c_n)(c_{n-2}-c_{n-1}))^3 (s_{n-1}s_{n-2})^3  = 1.
$$
Hence, $\phi^D$ is an algebra homomorphism. Also, $\phi^D$ preserves
the superalgebra structures since $\phi^D (t_i)$ and $t_i$ for each
$i$ are odd.

To show that $\phi^D:  \mathbb CD_n^- \rightarrow \Cl_{n}^0 \otimes
\C S_n$ is an isomorphism, it remains to verify the surjectivity as
both algebras have the same dimension. Equivalently, it suffices to
check that the generators $c_ic_n$ and $s_i$, for $1\le i\le n-1$,
of $\Cl_{n}^0 \otimes \C S_n$ lie in the image of $\phi^D$. To that
end, a direct computation shows that
\begin{align}
\phi^D(t_{n-1})\phi^D(t_{n}) &=c_{n-1}c_n,
  \notag \\
\phi^D(t_{i}) c_{i+1}c_n \phi^D(t_{i}) &=c_ic_n, \quad \text{ for }
1\le i\le n-2.
 \label{eq:tcct}
\end{align}
Inductively we conclude by \eqref{eq:tcct} that all $c_ic_n$, and
hence $\Cl_n^0$, lie in the image of $\phi^D$.
Now we can choose $x_i \in \C D_n^-$ such that $\phi^D (x_i) =
{\sqrt 2} (\zeta (c_i - c_{i+1}))^{-1} \in \Cl_n^0$, for $1\le i\le
n-1$. Then $\phi^D(x_i t_i) =\phi^D(x_i) \phi^D(t_i) =s_i$. Thus the
homomorphism $\phi^D$ is surjective. The theorem is proved.
\end{proof}

There is a natural inclusion  \cite[4.1]{KW3}
\begin{equation}  \label{eq:iota}
\iota: \C D_n^- \hookrightarrow \C B_n^-,
\end{equation}
which sends $t_i^D \mapsto t_i^B$ $(i\le n-1)$ and $t_n^D \mapsto
-t_n^B t_{n-1}^B t_n^B$, if we use superscripts to indicate the
types of Weyl groups. By Lemma~\ref{lem:zeta}, the superalgebra
$\langle\zeta\rangle$ generated by $\zeta$ is isomorphic to $\Cl_1$.
Recall $|A|$ denotes the underlying algebra for a superalgebra $A$.

\begin{prop}  \label{prop:iso:B2D}
We have an isomorphism of algebras:
\begin{align*}
|\C D_n^-| \times |\langle\zeta\rangle |
 & \stackrel{\cong}{\longrightarrow} |\C B_n^-|,
 \\
(x, \zeta^a) \mapsto \iota(x) \zeta^a, & \quad \text{ for } x\in |\C
D_n^-|, \, a=0,1.
\end{align*}
\end{prop}
Note that we cannot claim to have a superalgebra isomorphism $\C
D_n^- \times \langle\zeta\rangle \stackrel{\cong}{\longrightarrow}
\C B_n^-$, since the odd element $\zeta$ commutes but does not
super-commute with $\C D_n^-$.

\begin{proof}
If we identify $\C B_n^- \equiv\Cl_n \rtimes_- \C S_n^-$ as in
Theorem~\ref{CBn-iso}, we can naturally identify the subalgebra
$\iota(\C D_n) \equiv\Cl_n^0 \rtimes_- \C S_n^-$. Putting
Theorem~\ref{CBn-iso}, Lemma~\ref{lem:zeta} and Theorem~\ref{Dnodd
iso} together, we have the following commutative diagrams, where the
homomorphism $j$ extends the natural inclusion
$\Cl_n^0\hookrightarrow \Cl_n$ and sends $s_i \mapsto \zeta s_i$ for
each $i$:
\begin{eqnarray}  \label{eq:CD:BD}
\begin{CD}
\Cl_n^0 \rtimes_- \C S_n^- @>= >> \C D_n^-
  \\
 @VVV @V\iota VV
  \\
\Cl_n \rtimes_- \C S_n^- @>= >> \C B_n^-
\end{CD}
 \qquad\qquad
\begin{CD}
|\C D_n^-| @>\phi^D>> |\Cl_{n}^0 \otimes \C S_n|
  \\
@V\iota VV @Vj VV
  \\
|\C B_n^-| @>\phi^B>> |\Cl_{n} \otimes \C S_n|
\end{CD}
\end{eqnarray}
It follows that $\zeta \not\in \iota(\C D_n^-)$ and that $\zeta$
commutes with $\iota (\C D_n^-)$ . The proposition is proved.
\end{proof}

Recall $S^\la$ denotes the Specht module of $S_n$. Denote by $U^0$
the unique simple $\Cl_n^0$-module. Theorem~\ref{Dnodd iso} implies
immediately the following classification of simple $|\C
D_n^-|$-modules, which was obtained by Read \cite[Theorem 7.2]{Re2}
using a completely different construction of these simple modules.

\begin{cor}
Let $n$ be odd. A complete list of pairwise inequivalent simple $|\C
D_n^-|$-modules is $\{U^0\otimes S^\la \mid \la \vdash n\}$.
\end{cor}

\subsection{Split classes  for $n$ odd}

\begin{lem}  \label{Dnodd conjclasses}
Let $n$ be odd.
\begin{enumerate}
\item
The split conjugacy classes of $D_n$ are the classes of the
following types:
\begin{enumerate}
\item
$(\rho_+, \rho_-) \in (\OP, \EP)$, with $\ell(\rho_-)$ even;

\item
$(\rho_+, \rho_-) \in (\emptyset, \calP)$, with $\ell(\rho_-)$ even.
\end{enumerate}

\item
The split classes of type $(\rho_+, \rho_-) \in (\OP, \EP)$ are even
while those of type $(\rho_+, \rho_-)\in (\emptyset, \calP)$ are
odd.
\end{enumerate}
\end{lem}

\begin{proof}
(1) is \cite[Lemmas 6.4, 7.1]{Re2}. (2) follows by counting the
number of generators in a representative element of each conjugacy
class as given in \cite{Re2}, and noting that each generator $t_i$ of
$\C D_n^-$ is odd.
\end{proof}

\begin{remark}  \label{rem:DccBij}
Note that $\ell(\rho_+)$ must be odd in the case of Lemma~\ref{Dnodd
conjclasses}(1a). Hence the types of the split classes of $D_n$ are
in natural bijection with the partitions of $n$ by sending $(\rho_+,
\rho_-) \mapsto \rho_+\cup\rho_-$, so that the classes in (1a)
(respectively, (1b)) correspond to partitions of $n$ of odd lengths
(respectively, even lengths).
\end{remark}

\subsection{Counting the simples}
\label{sec:counting:odd}


\begin{lem}  \label{lem:count:nodd}
Let $n$ be odd.
\begin{enumerate}
\item
The number of simple $\C D_n^-$-modules of type $\typeM$ is equal to
all of 
\begin{equation*}  \label{eq:partNo}
({a})\; |\{\la\vdash n : n-\ell(\la) \mbox{ even}\}| - |\{\la\vdash
n : n-\ell(\la) \mbox{ odd}\}|; \quad
({b})\; |\mc{SOP}_n|; \quad
({c})\; |\mc{P}_n^{\rm{sym}}|.
\end{equation*}

\item
The number of simple $\C D_n^-$-modules of type $\typeQ$ is equal to
all of
\begin{equation*}  \label{eq:partNo2}
({a})\; |\{\la\vdash n : n-\ell(\la) \mbox{ odd}\}|; \quad
({b})\; |\{\{\la,\la'\} : \la'\neq\la \in\mc{P}_n\}|; \quad
({c})\; \frac12 (|\mc P_n| -|\mc{P}_n^{\rm{sym}}|).
\end{equation*}
\end{enumerate}
\end{lem}

\begin{proof}
Denote by $m$ (and respectively, $q$) the number of simple $\C
D_n^-$-modules of type $\typeM$ (and respectively, type $\typeQ$).
Then the number of simple $|\C D_n^-|$-modules is $m+2q$, which
is equal to $|\mc P_n|$ by counting the split classes in
Lemma~\ref{Dnodd conjclasses} and applying Wedderburn's
Theorem~\ref{wedderburn} and Remark~\ref{rem:ungraded}. From this we
easily see that (1) is consistent with (2), and so it suffices to
prove (1).

According to Proposition~ \ref{prop:splitcc} and Lemma~ \ref{Dnodd
conjclasses}, $m$ is given by (1a). It is a classical fact that
$(1a)=(1b) =(1c)$ (which is valid actually for any $n$). The
equality $(1b) =(1c)$ can be shown by an easy bijection, while the
equality $(1a)=(1b)$ can be established via a generating function
identity: $ \prod_{i\ge 1} \big(1+(-x)^i\big)^{-1} = \prod_{i\ge 1}
(1+x^{2i-1}). $ 
\end{proof}

\subsection{Classification of simple $\C S_n$-(super)modules}

Recall that $\chi^\la_\mu$ is the character of the Specht module
$S^\la$ on the conjugacy class of $S_n$ of cycle type $\mu$.
\begin{lem} \label{la and la'}
Let $n$ be odd. Then $\chi^\la_\mu = 0$ for all $\mu \vdash n$ of
even length if and only if $\la = \la'$.
\end{lem}

\begin{proof}
Note that
$$
\chi^{\la'} =\chi^\la \otimes \sgn, \qquad
\sgn_\mu=(-1)^{n-\ell(\mu)} =-(-1)^{\ell(\mu)},
$$
which we shall use repeatedly below.

$(\Leftarrow)$  Assume $\la = \la'$. If $\ell(\mu)$ is even, then
$\chi^\la_\mu = \chi^{\la'}_\mu =\chi^\la_\mu \cdot \sgn_\mu=  -
\chi^\la_\mu$, so $\chi^\la_\mu=0$.

$(\Rightarrow)$ Let $\nu, \mu \vdash n$ with $\ell(\mu)$ even, and
$\ell(\nu)$ odd.  Then $\chi^{\la'}_\mu = - \chi_\mu^{\la} =0=
\chi^\la_\mu$. Also, $\chi^{\la'}_\nu = \chi^\la_\nu \cdot \sgn_\nu=
(-1)^{n-\ell(\nu)}\chi^\la_\nu = \chi^\la_\nu$. Hence $\chi^\la =
\chi^{\la'}$ and so $\la = \la'$.
\end{proof}

\begin{prop}  \label{prop:superSn}
Let $\C S_n$ be endowed with the superalgebra structure with each
$s_i$ odd. Then a complete list of pairwise inequivalent simple $\C
S_n$-modules consists of:
\begin{enumerate}
\item
$S^\la$ of type $\typeM$, for $\la\vdash n$ with $\la=\la'$.

\item
$S^{\{\la,\la'\}}:=S^\la \oplus S^{\la'}$ of type $\typeQ$, for
pairs $\{\la,\la'\}$ with $\la\vdash n$ and $\la\neq\la'$.
\end{enumerate}
\end{prop}

\begin{proof}
Given a finite-dimension semisimple $\C$-superalgebra $A$, one
defines an involution $\alpha$ on $A$ by letting $\alpha
(a)=(-1)^{|a|} a$ for homogeneous $a\in A$. Given any  $|A|$-module
$N$, one obtains another $|A|$-module $N'$ with the same underlying
vector space as $N$ but with an action twisted by $\alpha$. It is
shown in \cite[Proposition~2.17]{Joz1} that
\begin{enumerate}
\item[(i)]
If $N$ is a simple $|A|$-module but not an $A$-module, then $N
\not\cong N'$ and $N\oplus N'$ can be endowed with a simple $A$-module
structure of type $\typeQ$.

\item[(ii)]
If a simple $|A|$-module $N$ can be lifted to a simple $A$-module
(which must be of type $\typeM$), then $N\cong N'$.
\end{enumerate}

In our setting with $A=\C S_n$, the simple $\C S_n$-modules are
$S^\la$, the involution is given by
$\alpha(\sigma)=(-1)^{l(\sigma)}\sigma$ for $\sigma\in S_n$, and the
twisted module $N'$ is isomorphic to $N\otimes \sgn$. Now (1)
follows from (ii) above and the fact that $S^\la \otimes \sgn \cong
S^{\la'}$. By Lemma~\ref{lem:count:nodd}, we have obtained all
simple $\C S_n$-modules of type $\typeM$. Then $S^\la$ for $\la \neq
\la'$ must pair with $S^{\la'}$ to give rise to simple $\C
S_n$-modules of type $\typeQ$, by applying (i) above. That $S^\la$
for $\la \neq \la'$ is not a $\C S_n$-module also follows from
Lemma~\ref{la and la'}, which says $S^\la$ has non-vanishing
character value on odd conjugacy classes.
\end{proof}

\subsection{Simple $\C D_n^-$-modules for $n$ odd}

Recall $U^0$ denotes the unique simple $\Cl_n^0$-module. We
introduce the following notation for $n$ odd:
\begin{equation}  \label{eq:Dla}
\D :=
\begin{cases}
 U^0 \otimes S^\la &
\mbox{ if } \la = \la',
 \\
U^0 \otimes (S^\la \oplus S^{\la'})
 &  \mbox{ if } \la \neq \la'.
\end{cases}
\end{equation}
By definition $\D$ only depends on the unordered pair
$\{\la,\la'\}$.
Via the isomorphism $\phi^D$ from Theorem~\ref{Dnodd iso}, $\D$ is a
$\C D_n^-$-module.

\begin{prop}\label{spinDnoddirreps}
Let $n$ be odd. Then $\{\D | \la\vdash n\}$ forms a complete set of
pairwise inequivalent simple $\mathbb CD_n^-$-modules. Moreover,
$\D$ is of type $\typeM$ if $\la = \la'$, and of type $\typeQ$
otherwise.
\end{prop}

\begin{proof}
Follows from Theorem~\ref{Dnodd iso} and
Proposition~\ref{prop:superSn}, by noting that $\Cl_n^0$ is purely
even and the unique simple $\Cl_n^0$-module $U^0$ is of type
$\typeM$.
\end{proof}

Recall the irreducible $\C B_n^-$-modules $B^\la$ for $\la \vdash
n$ from Proposition \ref{spinBnirreps}, which are all of type
$\typeQ$ since $n$ is odd. Recall from \eqref{eq:iota} the inclusion
$\C D_n^- \hookrightarrow \C B_n^-$.

\begin{prop}
 \label{prop:Bla:Dn}
Let $n$ be odd. Then
\begin{enumerate}
\item
As a $|\C D_n^-|$-module, $B^\la$ is a sum of two simple modules,
i.e., $B^\la\cong U^0 \otimes S^\la \oplus U^0 \otimes S^{\la'}$,
for $\la \vdash n$.

\item
$B^\la|_{\C D_n^-} \cong B^{\la'}|_{\C D_n^-} \cong\D$, for $\la
\vdash n$ with $\la \neq \la'$.

\item
$B^\la|_{\C D_n^-} \cong (\D)^{\oplus 2}$, for $\la \vdash n$ with
$\la=\la'$.
\end{enumerate}
\end{prop}

\begin{proof}
By construction, $B^\la \cong U\otimes S^\la$ via the isomorphism
$\phi^B$, where $U$ is the simple $\Cl_n$-module of type $\typeQ$.
$U$ decomposes into a sum of two copies of the $\Cl_n^0$-module
$U^0$, on which $\zeta$ acts by $\pm 1$ respectively. By the
(second) commutative diagram in \eqref{eq:CD:BD}, the action of
$s_i\in \C S_n \subseteq \phi^D(\C D_n^-)$ on $U\otimes S^\la$ is
twisted by the action of $\zeta$, giving rise to $U^0 \otimes S^\la
\oplus U^0 \otimes S^{\la'}$, whence (1).

Assume $\la \neq \la'$. Then $B^\la|_{\C D_n^-}$ must be isomorphic
to the simple ${\C D_n^-}$-module $\D$ by applying (1) and
Proposition~\ref{spinDnoddirreps}, whence (2).

Part~(3) follows  by applying (1) and
Proposition~\ref{spinDnoddirreps} again.
\end{proof}

\subsection{Spin fake degrees of $D_n$ for $n$ odd}

Recall $\B_{D_n}$ is the basic spin $\C D_n^-$-module. We would like
to compute
\begin{align} \label{eq:HPDn:odd}
\begin{split}
H^-_{D_n} (\la\la', t) &:= \displaystyle \sum_k \dim \Hom_{\C D_n^-}
\big(\D, \B_{D_n} \otimes S^kV \big) t^k,
    \\
P_{D_n}^-(\la\la', t) &:= \displaystyle \sum_k \dim \Hom_{\C D_n^-}
\big(\D, \B_{D_n} \otimes (S^kV)_{D_n}\big) t^k.
\end{split}
\end{align}

Proposition~\ref{prop:Bla:Dn} allows us to reduce the computations
to the type $B_n$ case.

\begin{thm}\label{spinDnodd mults}
Let $n$ be odd. Then
\begin{align*}
H_{D_n}^-( \la\la', t) =
 \begin{cases}
 \displaystyle  2 t^{2n(\la)}\prod_{\square\in
\la}\frac{1+t^{2c_\square+1}}{1-t^{2h_\square}}, &\text{ if } \la =
\la',
 \\
 \displaystyle  \frac{2t^{2n(\la)}}{\prod_{\square\in \la}
(1-t^{2h_\square})}
 \left( \prod_{\square\in\la} (1+t^{2c_\square +1})
+ \prod_{\square \in \la} ( t^{2c_\square} + t) \right), &
 \text{ if } \la \ne \la'.
 \end{cases}
\end{align*}
\end{thm}

\begin{proof}
Note that $\B_{D_n} \cong \B_{B_n}|_{\C D_n^-}$. So we have
$$
\B_{D_n} \otimes S^*V \cong (\B_{B_n} \otimes S^*V)|_{\C D_n^-} =
\bigoplus_{\la\vdash n} \frac12 H_{B_n}^-(\la, t)B^\la|_{\C D_n^-};
$$
here the factor $\frac12$ arises since by
Proposition~\ref{spinBnirreps} all $\C B_n^-$-modules $B^\la$ are
type $\typeQ$ when $n$ is odd. Also, by definition
$$
\B_{D_n}\otimes S^*V = \bigoplus_{\la=\la'}H_{D_n}^-(\la\la', t) \D
\oplus\bigoplus_{\{\la,\la'|\la \ne\la'\}}\frac 12
H_{D_n}^-(\la\la',t)\D.
$$
By \eqref{eq:Dla}, Proposition~\ref{prop:Bla:Dn} and a comparison of
the above two
expansions for $\B_{D_n} \otimes S^*V$, we have
\begin{align}  \label{eq:HDB-}
H_{D_n}^-(\la\la', t) =
 \begin{cases}
 \displaystyle  H_{B_n}^-( \la, t), &\text{ if } \la =
\la',
 \\
 \displaystyle  H_{B_n}^-( \la, t) +H_{B_n}^-( \la', t), &
 \text{ if } \la \ne \la'.
 \end{cases}
\end{align}
Recall the formula for $H_{B_n}^-( \la, t)$ in Theorem~\ref{thm:HBn-}.
We can rewrite
\begin{equation} \label{eq:Hla'}
H^-_{B_n}(\la',t)
 = 2^{p(n)} t^{2n(\la)}\prod_{\square\in
\la}\frac{t^{2c_\square}+t}{1-t^{2h_\square}},
\end{equation}
with $p(n)=1$, using the following identities:
$$
t^{2n(\la')}\prod_{\square\in \la'}(1+t^{2c_\square+1})
= \prod_{(i,j)\in \la'}(t^{2(i-1)}+t^{2(j-1)+1})
=t^{2n(\la)}\prod_{\square\in \la}(t^{2c_\square}+t).
$$
Now the theorem follows by applying to \eqref{eq:HDB-} the formulas
in Theorem~\ref{thm:HBn-} and \eqref{eq:Hla'}.
\end{proof}

The following  is equivalent to Theorem~\ref{spinDnodd mults} by
Lemma~\ref{lem:P=W} and using the well-known fact that the degrees
of $D_n$ are $2,4, \ldots, 2n-2, n$.

\begin{thm}  \label{PDn-}
Let $n$ be odd and $\la \vdash n$.  Then the spin fake degree
$P^-_{D_n}(\la\la', t)$ is
\begin{align*}
 \begin{cases}
 \displaystyle  2 t^{2n(\la)}\prod_{\square\in
\la}\frac{1+t^{2c_\square+1}}{1-t^{2h_\square}}\prod_{i=1}^{n-1}(1-t^{2i}) (1-t^n), &\text{ if } \la =\la',
 \\
 \displaystyle  \frac{2t^{2n(\la)}}{\prod_{\square\in \la}
(1-t^{2h_\square})}
 \Big( \prod_{\square\in\la} (1+t^{2c_\square +1})
+ \prod_{\square \in \la} ( t^{2c_\square} + t) \Big)\prod_{i=1}^{n-1}(1-t^{2i}) (1-t^n), &
 \text{ if } \la \ne \la'.
 \end{cases}
\end{align*}
\end{thm}

\begin{prop}\label{Dpalindromeodd}
The following palindromicity holds for spin fake degrees of $D_n$,
for all $\la\vdash n$:
$$
P^-_{D_n}(\la\la', t) = t^{n(n-1)}P^-_{D_n}(\la\la', t^{-1}).
$$
\end{prop}
\begin{proof}
Using Lemma~\ref{lem:P=W} and \eqref{eq:HDB-} while keeping in mind
the degrees for $B_n$ and $D_n$, we relate $P^-_{D_n}$ to
$P^-_{B_n}$ as follows:
\begin{align*}
P^-_{D_n}(\la\la', t) =
 \begin{cases}
\displaystyle  P_{B_n}^-(\la, t) \frac{(1-t^n)}{(1-t^{2n})}, &\text{
if } \la =\la',
 \\
\displaystyle \big(P_{B_n}^-(\la, t) + P_{B_n}^-(\la', t)\big)
\frac{(1-t^n)}{(1-t^{2n})}, &
 \text{ if } \la \ne \la'.
 \end{cases}
\end{align*}
Since $P^-_{B_n}(\la, t)$ and $P^-_{B_n}(\la', t)$ are palindromic
with the same shift number $n^2$, their sum will be as well.  It remains
palindromic upon multiplication by $\frac{(1-t^n)}{(1-t^{2n})}$,
 with a new shift number $(n^2-n)$.
\end{proof}

\subsection{Hecke-Clifford algebra $\aHC_{D_n}$ for $n$ odd}

Recall from Proposition~\ref{prop:Kla} that the simple
$\aHC_{B_n}$-modules are $K^\la= \Cl_n \otimes S^\la$ defined in
\eqref{eq:Kla} for $\la \vdash n$ and they are all of type $\typeM$.

\begin{lem}
Let $n$ be odd.  If $\la = \la'$, then $K^\la|_{\aHC_{D_n}}$ is a
type $\typeQ$ simple $\aHC_{D_n}$-module. Otherwise, if $\la \neq
\la'$, then $K^\la|_{\aHC_{D_n}}$  is a type $\typeM$ simple
$\aHC_{D_n}$-module.
\end{lem}

\begin{proof}
This follows from Propositions~\ref{isofinite}, \ref{functors},
\ref{spinDnoddirreps} and \ref{prop:Bla:Dn}.
\end{proof}

Denote
\begin{align*}
H_{D_n}(\la\la', t) &:= \displaystyle \sum_k \dim \Hom_{\aHC_{D_n}}
(K^\la|_{\aHC_{D_n}}, \Cl_\h \otimes S^kV) t^k,
  \\
P_{D_n}(\la\la', t) &:= \displaystyle \sum_k \dim \Hom_{\aHC_{D_n}}
(K^\la|_{\aHC_{D_n}}, \Cl_\h \otimes (S^kV)_{D_n}) t^k .
\end{align*}

Recall the formulas for $H_{D_n}^-( \la\la', t)$ from
Theorem~\ref{spinDnodd mults}, and the formulas for $P_{D_n}^-(
\la\la', t)$ from Theorem~\ref{PDn-}. The following proposition
allows us to compute closed formulas for $H_{D_n}(\la\la', t)$ and
$P_{D_n}(\la\la', t)$.

\begin{prop} \label{H_Dn mults}
Let $n$ be odd.  If $\la = \la'$, then
$$
H_{D_n}(\la\la', t) = H_{D_n}^-( \la\la', t),
 \quad
P_{D_n}(\la\la', t)  = P_{D_n}^-(\la\la', t).
$$
If $\la \ne \la'$, then
$$
H_{D_n}(\la\la', t) =\frac 12H_{D_n}^-( \la\la', t), \quad
P_{D_n}(\la\la', t)  = \frac 12 P_{D_n}^-(\la\la', t).
$$
\end{prop}

\begin{proof}
This follows from Proposition \ref{prop:equiv} and Lemma~\ref{lem:P=W}.
\end{proof}

\section{The spin fake degrees of type $D_n$ (for $n$ even)}

\label{sec:Dn:even}

Let $n$ be even throughout this section.

\subsection{The algebra $\C D_n^-$ for $n$ even}


\begin{lem} (cf. \cite{Re2})
  \label{Dneven conjclasses}
Let $n$ be even.
\begin{enumerate}
\item
The split classes of $\C D_n^-$ are the classes of cycle types
$(\rho_+,\rho_-)$ in $(\OP, \EP)$ or in $(\emptyset, \SOP)$, with
$\ell(\rho_-)$ even.

\item All split classes are even.
\end{enumerate}
\end{lem}

\begin{proof}
Part (1) was proved in \cite[Lemma~ 6.4]{Re2}. Since all generators
$t_i$ are odd, the parity of an element $t_{i_1}\cdots t_{i_k}$, and
thus of its conjugacy class, is equal to the parity of $k$. Now (2)
can be extracted from \cite{Re2}.
\end{proof}

By  \cite[Lemma 8.1]{Re2}, the number of these split classes in $D_n$
is equal to the number of conjugacy classes of the alternating group
$A_n$. In the spirit of Theorem~ \ref{Dnodd iso}, the work of Read
suggests the following structure result for $\C D_n^-$.

\begin{conj}  \label{conj=Dn}
For $n$ even, we have a superalgebra isomorphism $\Cl_n \otimes
\mathbb CA_n \cong \C D_n^-$, where $\C A_n$ is even while the $n$
generators of $\Cl_n$ are odd.
\end{conj}

\subsection{Simple $\C D_n^-$-modules for $n$ even}

Recall the simple $\C B_n^-$-modules $B^\la$ from Proposition~
\ref{spinBnirreps}, which are all of type $\typeM$ since $n$ is
even.
Read \cite[Lemma 8.4, Corollary 8.12]{Re2} has classified the
ungraded irreducible $\C D_n^-$-modules. Recall from
Proposition~\ref{spinBnirreps} the simple $\C B_n^-$-modules
$B^\la$, for $\la\vdash n$, and set
\begin{equation} \label{eq:Dla:ev}
\D := B^\la|_{\C D_n^-}, \qquad (\text{for $n$ even}).
\end{equation}
According to Read, $|\D|$ is a simple $|\C D_n^-|$-module if
$\la\neq \la'$.
In the case when $\la=\la'$, $\D$ is a sum of two inequivalent simple
$|\C D_n^-|$-modules $D^\la_{\pm}$:
\begin{equation} \label{eq:Dla:ev2}
\D=D^\la_+\oplus D^\la_-.
\end{equation}

All these simple $|\C D_n^-|$-modules can be endowed with the
structures of simple $\C D_n^-$-modules by
Remark~\ref{rem:ungraded}, since the graded irreducibles are all of
type $\typeM$, by Proposition~\ref{prop:splitcc} and
Lemma~\ref{Dneven conjclasses}. Hence, Read's results can be
upgraded as follows.

\begin{prop} \label{spinDnevenirreps}
Let $n$ be even. A complete list of pairwise inequivalent simple $\C
D_n^-$-modules consists of $\D$ when $\la \ne \la'$ and $D^\la_\pm$
when $\la = \la'$, for $\la\vdash n$. All these simple $\C
D_n^-$-modules are of type $\typeM$.
\end{prop}
It is well known that the simple $A_n$-modules are parametrized in
the same way as the simple $\C D_n^-$-modules above. Read's classification
of the simple $|\C D_n^-|$-modules 
can be reformulated as stating that the ungraded version of the  
isomorphism in Conjecture~\ref{conj=Dn} holds. On the other hand,
Conjecture~\ref{conj=Dn}, if established by a direct and
constructive proof, would immediately provide a new proof of the classification of simple
$\C D_n^-$-modules.

\subsection{Spin fake degrees of $D_n$ for $n$ even}
\label{subsec:degDn:ev}

We continue the notation $H_{D_n}^-(\la\la', t)$ and
$P_{D_n}^-(\la\la', t)$ from \eqref{eq:HPDn:odd} when $\la \neq
\la'$, now for $n$ even. In addition, we will write
$H_{D_n}^-(\la_\pm, t)$ and $P_{D_n}^-(\la_\pm, t)$ to indicate the
graded multiplicities of the simple modules $D^\la_\pm$ in $\B_{D_n}
\otimes S^*V$ and $\B_{D_n} \otimes (S^*V)_{D_n}$, when $\la =
\la'$, for $n$ even.

\begin{thm} \label{spinDneven mults}
Let $n$ be even. Then we have
\begin{align*}
 \begin{cases}
 \displaystyle
H_{D_n}^-(\la_+, t) = H_{D_n}^-(\la_-, t)  =
t^{2n(\la)}\prod_{\square\in
\la}\frac{1+t^{2c_\square+1}}{1-t^{2h_\square}}, &\text{ if } \la =
\la',
 \\
 \displaystyle
 H_{D_n}^-( \la\la', t) = \frac{t^{2n(\la)}}{\prod_{\square\in \la}
(1-t^{2h_\square})}
 \left( \prod_{\square\in\la} (1+t^{2c_\square +1})
+ \prod_{\square \in \la} ( t^{2c_\square} + t) \right), &
 \text{ if } \la \ne \la'.
 \end{cases}
\end{align*}
\end{thm}

\begin{proof}
As in the proof of Theorem \ref{spinDnodd mults}, we write
$$
\B_{D_n} \otimes S^*V \cong (\B_{B_n} \otimes S^*V)|_{\C D_n^-} =
\bigoplus_{\la\vdash n} H_{B_n}^-(\la, t)B^\la|_{\C D_n^-},
$$
since all $B^\la$ are type $\typeM$ when $n$ is even. On the other
hand, it follows by definition that
$$
\B_{D_n} \otimes S^*V  = \bigoplus_{\{\la \neq \la'\}}
H_{D_n}^-(\la\la',t) \D \oplus \bigoplus_{\la =\la'}
\big(H_{D_n}^-(\la_+, t) D^\la_+ \oplus H_{D_n}^-(\la_-, t)
D^\la_-\big).
$$
A comparison of the above two identities using \eqref{eq:Dla:ev} and
\eqref{eq:Dla:ev2} gives us
\begin{equation*}
\begin{cases}
H_{D_n}^-(\la\la',t) =H_{B_n}^-(\la,t)+H_{B_n}^-(\la',t), & \text{
if } \la \neq \la',
  \\
H_{D_n}^-(\la_\pm, t) =H_{B_n}^-(\la,t), & \text{ if }\la = \la'.
\end{cases}
\end{equation*}
Now the theorem follows from these two identities, the formula for
$H_{B_n}^-( \la, t)$ in Theorem~\ref{thm:HBn-}, and the identity
\eqref{eq:Hla'} with $p(n)=0$.
\end{proof}

The following is equivalent to Theorem~\ref{spinDneven mults} by
Lemma~\ref{lem:P=W} and using the fact that the degrees of $D_n$ are
$2,4, \ldots, 2n-2, n$.

\begin{thm}   \label{PDn-:even}
Let $n$ be even. Then the spin fake degree $P^-_{D_n}(\la\la', t)$
is
\begin{align*}
 \begin{cases}
 \displaystyle  P_{D_n}^-(\la_\pm,t) = t^{2n(\la)}\prod_{\square\in
\la}\frac{1+t^{2c_\square+1}}{1-t^{2h_\square}}\prod_{i=1}^{n-1}(1-t^{2i})
(1-t^n), &\text{ if } \la =\la',
 \\
 \displaystyle  P_{D_n}^-(\la\la',t) = &
  \\
 \displaystyle \quad \frac{t^{2n(\la)}}{\prod_{\square\in \la}
(1-t^{2h_\square})}
 \Big( \prod_{\square\in\la} (1+t^{2c_\square +1})
+ \prod_{\square \in \la} ( t^{2c_\square} + t)
\Big)\prod_{i=1}^{n-1}(1-t^{2i}) (1-t^n), &
 \text{ if } \la \ne \la'.
 \end{cases}
\end{align*}
\end{thm}

\begin{prop} \label{Dpalindromeeven}
The following palindromicity holds for spin fake degrees of $D_n$:
for each irreducible $\C D_n^-$-character $\chi$, we have
$
P^-_{D_n}(\chi, t) = t^{n(n-1)}P^-_{D_n}(\chi, t^{-1}).
$
\end{prop}
\begin{proof}
The proof is similar to that of Proposition \ref{Dpalindromeodd}.
\end{proof}

\subsection{Hecke-Clifford algebra $\aHC_{D_n}$ for $n$ even}

Recall from Proposition~\ref{prop:Kla} that the simple
$\aHC_{B_n}$-modules $K^\la$ are parametrized by $\la \vdash n$ and
are all of type $\typeM$.

\begin{prop}  \label{prop:HCDn:even}
Let $n$ be even.

\begin{enumerate}
\item
If $\la \ne \la'$, then $K^\la|_{\aHC_{D_n}}$ is a simple
$\aHC_{D_n}$-module, and $K^\la|_{\aHC_{D_n}} \cong
K^{\la'}|_{\aHC_{D_n}}$.

\item
If $\la = \la'$, then $K^\la|_{\aHC_{D_n}}$ is a sum of two
inequivalent simple $\aHC_{D_n}$-modules. All these simple modules
are of type $\typeM$.

\item
The simple modules in (1) and (2) (modulo the identifications in
(1)) form a complete list of inequivalent simple
$\aHC_{D_n}$-modules.
\end{enumerate}
\end{prop}

\begin{proof}
By Proposition~\ref{isofinite}, there is a Morita super-equivalence
between $\C D_n^-$ and $\aHC_{D_n}$, and hence
Proposition~\ref{functors} applies. By Lemma~\ref{lem:K=B}, $K^\la$
corresponds to $B^\la$ under the super-equivalence. Now the
proposition follows.
\end{proof}
The list of simple $\aHC_{D_n}$-modules in
Proposition~\ref{prop:HCDn:even} corresponds bijectively via the
Morita super-equivalence to the list of simple $\C D_n^-$-modules in
Proposition~\ref{spinDnevenirreps}. Proposition~\ref{prop:equiv}
ensures that the graded multiplicities of a simple
$\aHC_{D_n}$-module in $\Cl_V \otimes S^*V$ (respectively, in $\Cl_V
\otimes (S^*V)_{D_n}$) are exactly the same as their counterparts in
$\B_{D_n} \otimes S^*V$ (respectively, in $\B_{D_n} \otimes
(S^*V)_{D_n}$) given in Section~\ref{subsec:degDn:ev}.

\vspace{.3cm}

{\bf Acknowledgments.}
The first author gratefully acknowledges the support of a UVA semester fellowship.
The second author is partially supported by
an NSF grant DMS--1101268. We thank the referee for careful reading and helpful comments.
  

\end{document}